\documentclass[11 pt]{amsart}

\usepackage{amsmath}             % AMS-LaTeX
\usepackage{amssymb}             % AMS-LaTeX
\usepackage{latexsym}            % LaTeX special symbols
\usepackage[active]{srcltx}
\usepackage{ae}
\numberwithin{figure}{section}

\numberwithin{figure}{section}

\newtheorem{theorem}{Theorem}[section]
\newtheorem{lemma}[theorem]{Lemma}

\theoremstyle{definition}
\newtheorem{definition}[theorem]{Definition}

\newtheorem{remark}[theorem]{Remark}
\numberwithin{equation}{section}
\newcommand{\norm}[1]{\left|\left|#1\right|\right|}

\def\vint_#1{\mathchoice%
          {\mathop{\kern 0.2em\vrule width 0.6em height 0.69678ex depth -0.58065ex
                  \kern -0.8em \intop}\nolimits_{\kern -0.4em#1}}%
          {\mathop{\kern 0.1em\vrule width 0.5em height 0.69678ex depth -0.60387ex
                  \kern -0.6em \intop}\nolimits_{#1}}%
          {\mathop{\kern 0.1em\vrule width 0.5em height 0.69678ex depth -0.60387ex
                  \kern -0.6em \intop}\nolimits_{#1}}%
          {\mathop{\kern 0.1em\vrule width 0.5em height 0.69678ex depth -0.60387ex
                  \kern -0.6em \intop}\nolimits_{#1}}}
\def\vintslides_#1{\mathchoice%
          {\mathop{\kern 0.1em\vrule width 0.5em height 0.697ex depth -0.581ex
                  \kern -0.6em \intop}\nolimits_{\kern -0.4em#1}}%
          {\mathop{\kern 0.1em\vrule width 0.3em height 0.697ex depth -0.604ex
                  \kern -0.4em \intop}\nolimits_{#1}}%
          {\mathop{\kern 0.1em\vrule width 0.3em height 0.697ex depth -0.604ex
                  \kern -0.4em \intop}\nolimits_{#1}}%
          {\mathop{\kern 0.1em\vrule width 0.3em height 0.697ex depth -0.604ex
                  \kern -0.4em \intop}\nolimits_{#1}}}

\newcommand{\kint}{\vint}

\newcommand{\aveint}[2]{\mathchoice%
          {\mathop{\kern 0.2em\vrule width 0.6em height 0.69678ex depth -0.58065ex
                  \kern -0.8em \intop}\nolimits_{\kern -0.45em#1}^{#2}}%
          {\mathop{\kern 0.1em\vrule width 0.5em height 0.69678ex depth -0.60387ex
                  \kern -0.6em \intop}\nolimits_{#1}^{#2}}%
          {\mathop{\kern 0.1em\vrule width 0.5em height 0.69678ex depth -0.60387ex
                  \kern -0.6em \intop}\nolimits_{#1}^{#2}}%
          {\mathop{\kern 0.1em\vrule width 0.5em height 0.69678ex depth -0.60387ex
                  \kern -0.6em \intop}\nolimits_{#1}^{#2}}}

\newcommand{\eps}{\varepsilon}
\newcommand{\R}{\mathbb{R}}

\renewcommand{\limsup}{\operatornamewithlimits{lim \, sup}}

\newcommand{\esssup}{\operatornamewithlimits{ess\, sup}}

\renewcommand{\l}{\left}
\renewcommand{\r}{\right}
\newcommand{\parts}[2]{\frac{\partial {#1}}{\partial {#2}}}

\numberwithin{equation}{section}

\newcommand{\trm}{\textrm}
\newcommand{\half}{{\frac{1}{2}}}
\newcommand{\q}[1]{Q_{#1}}

\newcommand{\Om}{\Omega}

\newcommand{\N}{N^{1,2}}

\newcommand{\limminus}{{\mathchoice{\raise.17ex\hbox{$\scriptstyle -$}}
                {\raise.17ex\hbox{$\scriptstyle -$}}
                {\raise.1ex\hbox{$\scriptscriptstyle -$}}
                {\scriptscriptstyle -}}}
\newcommand{\limplus}{{\mathchoice{\raise.17ex\hbox{$\scriptstyle +$}}
                {\raise.17ex\hbox{$\scriptstyle +$}}
                {\raise.1ex\hbox{$\scriptscriptstyle +$}}
                {\scriptscriptstyle +}}}

\newcommand{\abs}[1]{\left| #1 \right|}
\newcommand{\Lip}{\operatorname{Lip}}
\newcommand{\beq}{\begin{equation}}
\newcommand{\eeq}{\end{equation}}

\begin{document}

\title[Global higher integrability of parabolic quasiminimizers]{Global higher integrability for parabolic quasiminimizers in metric spaces}

\author{Mathias Masson and Mikko Parviainen}
\address{Mathias Masson,
Aalto University, School of Science and Technology, Department of Mathematics,
PO Box 11000, FI-00076 Aalto,
Finland}
\email{mathias.masson.finland@gmail.com}

\address{Mikko Parviainen,
Department of Mathematics and Statistics, P.O.Box (MaD) FI-40014 University of Jyväskylä, 
Finland
}
\email{mikko.j.parviainen@jyu.fi}

\thanks{}



\begin{abstract}
We prove higher integrability up to the boundary for minimal $p$-weak upper gradients of parabolic quasiminimizers in metric measure spaces, related to the heat equation. We assume the underlying metric measure space to be equipped with a doubling measure and to support a weak Poincaré-inequality.
\end{abstract}
\maketitle

\section{Introduction}

The problem of finding a solution to the classical heat equation
\begin{align*}
-\frac{\partial u}{\partial t}+\Delta u=0,
\end{align*}
 in a parabolic cylinder $\Omega\times (0,T)$ can be reformulated into the variational problem of finding a function $u$ such that with $K=1$ we have
\begin{equation}\label{intro parabolic quasiminimizer}
\begin{split}
2\int_{\{\phi\neq 0\}}u\parts{\phi}{t} \,d x \,d t + &\int_{\{\phi\neq 0\}}
\abs{\nabla u}^2
 \,d x \,d t\\
 &\leq K\int_{ \{\phi\neq 0\}}
 \abs{\nabla(u+\phi)}^2\,d x \,d t, 
\end{split}
\end{equation}
for all compactly supported test functions $\phi\in C_0^{\infty}(\Omega\times (0,T))$. Here $\Omega$ denotes a bounded domain in $\R^d$. A generalization of this minimization problem is to consider inequality \eqref{intro parabolic quasiminimizer} with the relaxed assumption $K\geq 1$: a function $u$ satisfying this generalized condition is then called a parabolic quasiminimizer \cite{wieser87} related to the heat equation. 

Our main result, Theorem \ref{thm:global-high-int}, is to show that if $u$ is a parabolic quasiminimizer  in the general metric measure space setting, and satisfies a Dirichlet type parabolic boundary condition, where the domain is assumed to be regular enough, then $u$ has the following global higher integrability property:  The upper gradient \cite{HeinKosk98} of $u$ is  integrable over the whole cylinder $\Omega\times (0,T)$ to a slightly higher power than initially assumed.

Assuming a weak Poincaré inequality, a doubling measure and a thickness condition for the complement of the domain $\Omega$, we  prove a parabolic Poincaré and Caccioppoli type estimate for $u$ up to the boundary. Then we combine these with Sobolev's inequality and a self improving property for the thickness condition \cite{Lewi88},\cite{BjorMacmShan01} to establish a reverse Hölder inequality up to the boundary.

The novelty of this paper is that we prove these estimates in the general metric measure space setting, using a purely variational approach. No reference is made to the concept of a weak solution, or to the explicit scaling properties of the measure. Furthermore, no assumptions of translation invariance or absolute continuity of the underlying measure are made. Instead we base the proofs on taking integral averages and on the doubling property of the measure. On the other hand, the concept of parabolic quasiminimizers is extended into metric spaces by replacing gradients with the more general concept of upper gradients, which do not require the existence of partial derivatives.

The concept of quasiminimizers originates from the elliptic case in Euclidean spaces, where Giaquinta and Giusti showed in their celebrated papers \cite{giaquintag82, giaquintag84} that many properties of weak solutions to elliptic PDEs generalize to a class of elliptic quasiminimizers
\[
\begin{split}
\int_{\{\phi\neq 0\}} \abs{\nabla u}^2 \,d x\le K \int_{\{\phi\neq 0\}} \abs{\nabla (u+\phi)}^2 \,d x.
\end{split}
\]
Hence, to some extent quasiminimizers provide a unifying approach to the theory of elliptic nonlinear partial differential equations. Indeed, for example a solution $u$ to 
$$
\operatorname{div}F(\nabla u)=0
$$ 
under suitable regularity assumptions and the growth bounds 
$$\alpha \abs{\nabla u}^2 \le F(\nabla u)\le \beta \abs{\nabla u}^2,$$
 is a quasiminimizer with a constant $K=\beta/\alpha$. However, in other respects, the theory of quasiminimizers differs from that of minimizers. For example, quasiminimizers do not provide a unique solution to the Dirichlet problem, and they do not obey the  comparison principle. One advantage of quasiminimizers is that they allow for replacing the gradients with a comparable concept which is definable in a more general setting. This way, by using upper gradients, Kinnunen and Shanmugalingam \cite{KinnShan01} were able to extend the concept of quasiminimizers to metric measure spaces.

Following Giaquinta and Giusti, parabolic quasiminimizers were introduced in the Euclidean setting by Wieser \cite{wieser87}.  In recent papers \cite{KinnMaroMiraParo12}, \cite{MassSilj11}, \cite{MaroMass12}, \cite{MassMiraParoParv12}, following Kinnunen and Shanmugalingam, the definition and study of parabolic quasiminimizers has been extended to metric measure spaces. In this paper we follow the same approach. 
%However, an advantage of the above variational formulation is that unlike the notion of a weak solution, it can readily be generalized to metric measure spaces of the type $ X \times (0,T)$, and even to systems. Our objective is to study regularity theory for \emph{parabolic quasiminimizers}  (Definition \ref{qmin}) satisfying
%\[
%\begin{split}
%2\int_{\{ \phi\neq 0\}}u\parts{\phi}{t} \,d x \,d t + \int_{\{ \phi\neq 0 \}}
%\abs{\nabla u}^2
% \,d x \,d t\leq K \int_{\{\phi\neq 0 \}}
%& \abs{\nabla(u+\phi)}^2\,d x \,d t,
%\end{split}
%\]
%where the modulus of the gradient is replaced by an upper gradient \cite{HeinKosk98}.

%Some difficulties are already present in the Euclidean setting. For example, absorption of the upper gradient terms in the energy estimates requires Widman's \cite{widman71} hole filling argument. Other problems only appear in the metric spaces. There is a well recognized difficulty in the parabolic setting: One often needs a test function depending on a solution $u$ itself, but $u$ does not necessarily have time derivatives. However, the standard mollification procedure to treat this problem seems to introduce an additional difficulty in the metric space setting. 

Substantial progress was made in the mid-1950s and -1960s in the
regularity theory of elliptic equations due to the discoveries of
De Giorgi \cite{degiorgi57}, Nash \cite{nash58} and Moser
\cite{moser60,moser61}.  A natural question was, whether these results extend to systems as well.  Morrey \cite{morrey68} proved that up to a set of measure zero a solution to a elliptic systems is regular. However, it was soon discovered by De Giorgi \cite{degiorgi68} followed by Giusti and Miranda \cite{giustim68}, that full regularity for systems actually fails, and thus the partial regularity is best one can, in general, hope for.

The generalizations of Morrey's partial regularity result (Giaquinta and Giusti \cite{giaquintag78} as well as Giaquinta and Modica \cite{giaquintam79}) rely on the higher integrability of the gradient.  Such results  for elliptic PDEs were obtained by  Bojarski \cite{bojarski57} as well as Meyers \cite{meyers63}, and by Gehring \cite{gehring73} in the context of quasiconformal mappings.   In \cite{elcm75}, Elcrat and
Meyers proved the local higher integrability for nonlinear
elliptic systems. Later, in \cite{giaqs82}, Giaquinta and Struwe studied similar questions for
systems of parabolic equations with quadratic growth conditions, and in \cite{kinnunenl00} Kinnunen and Lewis showed that $p$-parabolic type systems share the higher integrability property as well.  

Another natural direction to extend regularity results is to consider regularity up to the boundary. Already elliptic examples in \cite{kilpelainenk94} demonstrate that both the regularity of the boundary as well as the boundary values play a role in the proofs. Recently, local and global higher integrability questions have inspired an extensive literature, see for example \cite{granlund82}, \cite{wieser87}, \cite{arkhipova95}, \cite{Misa06}, \cite{AcerMing07},\cite{parviainen08},  \cite{bogelein08}, \cite{bogelein09}, \cite{parviainen09}, \cite{parviainen09b}, \cite{bogeleinp10}, \cite{byunr10}, \cite{byunrw10}, \cite{bogeleindm11}, \cite{fugazzola12}, and  \cite{habermann}.

\section{Preliminaries}

\subsection{Doubling measure}

Let
$X=(X,d,\mu)$ be a
complete linearly locally convex metric space endowed with a positive 
doubling Borel measure $\mu $ which supports a weak $(1,2)$-Poincaré inequality. %Assume also that $\Om \subset X$ is
%nonempty and open.

The measure $\mu$ is called \emph{doubling} if there exists a constant
$c_\mu \geq 1$, such that for all balls $B= B(x_0,r):=\{x\in X:
d(x,x_0)<r\}$ in~$X$,
\begin{equation*}
        \mu(2B) \le c_\mu \mu(B),
\end{equation*}
where $\lambda B=B(x_0,\lambda r)$. By iterating the doubling condition, it follows with
$s=\log_{2}c_{\mu}$ and $C=c_\mu^{-2}$ that
\begin{equation} \label{eq:DoublingConsequence}
\frac{\mu(B(z,r))}{\mu(B(y,R))} \geq C\Bigl(\frac{r}{R}\Bigr)^{s},
\end{equation}
for all balls $B(y,R)\subset X$, $z \in B(y,R)$
and $0 < r \leq R < \infty$.
However, the choice $s=\log_2 c_\mu$ may not be optimal,
and we just assume that $s$ is any number such that
\eqref{eq:DoublingConsequence} is satisfied. From now on, throughout this text we assume that $c_\mu>1$ and so $s>0$.

A metric space $X$ is called linearly locally convex if there exists constants $C_1>0$ and $r_1>0$ such that for all balls $B$ in $X$ with radius at most $r_1$, every pair of distinct points in the annulus $2B\setminus \overline B$ can be connected by a curve lying in the annulus $2C_1B\setminus C_1^{-1}\overline B$, see Section 3.12 in \cite{HeinKosk98} and \cite{BjorMacmShan01}. The assumption that $X$ is linearly locally convex will be needed for Theorem \ref{self improving thickness} below.

\subsection{Notation}

Next we introduce more notation  used throughout this paper. % Let
%$\Om_T=\Omega\times(0,T)$ where $\Omega\subset X$ is a domain and $0<T<\infty$. 
 Given  any $z_0=(x_0,t_0)\in X\times \R \trm{ and
}{\rho}>0,$ let
\[
B_\rho(x_0)=\{\,x\in X\, :\, d(x,x_0)<\rho\,\},
\]
denote an open ball in $X$, and let
\[
\Lambda_{\rho}(t_0)=(t_0-\half {\rho}^2,\, t_0+\half {\rho}^2),
\]
denote an open interval in $\R$. A space-time cylinder in
$X\times \R$ is denoted by
\[
\q{{\rho}}(z_0)=B_{\rho}(x_0)\times \Lambda_{\rho}(t_0),
\]
so that $\nu(\q{{\rho}}(z_0))=\mu(B_{\rho}(x_0))\rho^2$.
When no confusion arises, we shall omit the reference points
  and write briefly $B_{\rho},\ \Lambda_{\rho}$ and $Q_{\rho}$.
 We denote the product measure by $
d \nu=d \mu\, d t$. The integral average of $u$ is
denoted by
\begin{align}\label{spatial mean value}
u_{B_\rho}(t)=\kint_{B_{\rho}} u(x,t)\, d \mu=\frac{1}{\mu({B_{\rho})}}\int_{B_{\rho}} u(x,t)\, d \mu
\end{align}
and
\[
\begin{split}
\kint_{Q_\rho} u \,d \nu =\frac{1}{\nu(Q_\rho)}\int_{Q_\rho} u \,d \nu.
\end{split}
\]

%\begin{definition}
%A nonnegative Borel function $g$ on $X$ is an \emph{upper gradient}
%of an extended real-valued function $f$
%on $X$ if for all paths  $\gamma: [0,l_{\gamma}] \to X$,
%\begin{equation} \label{ug-cond}
%        |f(\gamma(0)) - f(\gamma(l_{\gamma}))| \le \int_{\gamma} g\,ds.
%\end{equation}
%If $g$ is a nonnegative measurable function on $X$
%and if (\ref{ug-cond}) holds for \p-a.e.\ path,
%then $g$ is a \emph{\p-weak upper gradient} of~$f$.
%
%By saying that (\ref{ug-cond}) holds for \p-a.e.\ path, we
%mean that it fails only for a path family with zero \p-modulus. It is
%implicitly assumed that $\int_{\gamma} g\,ds$ is defined (with a
%value in $[0,\infty]$) for \p-a.e.\ path.
%\end{definition}
\subsection{Upper gradients.} Following \cite{HeinKosk98}, a non-negative Borel measurable function $g: X \rightarrow [0, \infty]$ is said to be an upper gradient of a function $u: X\rightarrow [-\infty, \infty]$, if for all compact rectifiable paths $\gamma$ joining $x$ and $y$ in $X$ we have 
\begin{align}\label{upper gradient}
|u(x)-u(y)|\leq \int_{\gamma} g \,ds.
\end{align}
In case $u(x)=u(y)=\infty$ or $u(x)=u(y)=-\infty$, the left side is defined to be $\infty$. Assume $1\leq p <\infty$. The $p$-modulus of a family of paths $\Gamma$ in $X$ is defined to be
\begin{align*}
 \inf_\rho \int_X \rho^p \,d\mu,
\end{align*}
where the infimum is taken over all non-negative Borel measurable functions $\rho$ such that for all rectifiable paths $\gamma$ which belong to $\Gamma$, we have
\begin{align*}
 \int_{\gamma} \rho \, ds \geq 1.
\end{align*}
A property is said to hold for $p$-almost all paths, if the set of non-constant paths for which the property fails is of zero $p$-modulus. Following \cite{KoskMacm98,Shan00}, if \eqref{upper gradient} holds for $p$-almost all paths $\gamma$ in $X$, then $g$ is said to be a $p$-weak upper gradient of $u$. 

When $1<p<\infty$ and $u\in L^p(X)$, it can be shown \cite{Shan01} that  there exists a minimal $p$-weak upper gradient of $u$, we denote it by $g_u$, in the sense that $g_u$ is a $p$-weak upper gradient of $u$
and for every $p$-weak upper gradient $g$ of $u$ it holds $g_u\leq g$ $\mu$-almost everywhere in $X$. %The minimal $p$-weak upper gradient is unique up to a set of zero $\mu$-measure. 
Moreover, if $v=u$ $\mu$-almost everywhere in a Borel set $A\subset X$, then $g_v=g_u$ $\mu$-almost everywhere in $X$. Also, if $u,v\in L^p(X)$, then $\mu$-almost everywhere in $X$, we have
\begin{align*}
&g_{u+v}\leq g_u+g_v,\\
&g_{uv}\leq |u|g_v+|v|g_u.
\end{align*}
Proofs for these properties and more on upper gradients in metric spaces can be found for example in \cite{BjorBjor11} and the references therein. See also \cite{Chee99} for a discussion on upper gradients.

\subsection{Newtonian spaces.} Following \cite{Shan00}, for  $1<p<\infty$, and $u\in L^p(\Omega)$ where $\Omega\subset X$ is a domain, we define
\begin{align*}
\|u\|^p_{1,p,\Omega}=\|u\|^p_{L^p(\Omega,\mu)}+\|g_u\|^p_{L^p(\Omega,\mu)},
\end{align*}
and
\begin{align*}
\widetilde N^{1,p}(\Omega)= \{ u\,:\, \|u\|_{1,p,\Omega}<\infty\}.
\end{align*}
An equivalence relation in $\widetilde N^{1,p}(\Omega)$ is defined by saying that $u\sim v$ if
\begin{align*}
 \|u-v\|_{\widetilde N^{1,p}(\Omega)}=0.
\end{align*}
The \emph{Newtonian space} $N^{1,p}(\Omega)$ is defined to be the space $\widetilde N^{1,p}(\Omega)/ \sim$, with the norm
\begin{align*}
 \|u\|_{N^{1,p}(\Omega)}=\|u\|_{1,p,\Omega}.
\end{align*}

A function $u$ belongs to the local Newtonian space $N_{\textrm{loc}}^{1,p}(\Omega)$ if it belongs to $N^{1,p}(\Omega')$ for every $\Omega' \subset \subset \Omega$. The Newtonian space with zero boundary values is defined as $N_0^{1,p}(\Omega)=\{\, f\in N^{1,p}(\Omega)\,:\,f$ can be continued into a function in $N^{1,p}(X)$ by setting $f=0$ outside $\Omega\,\}$. For more properties of Newtonian spaces, see \cite{Hein01,Shan00, BjorBjor11}.

\subsection{Poincar\'e's and Sobolev's inequality}
For $1\leq q <\infty$, $1< p < \infty$, the measure $\mu$ is said to support a weak $(q,p)$-Poincar\'e inequality if there exist constants $c_P>0$ and $\lambda\ge 1$ such that
\begin{equation}
\l(\vint_{B_\rho(x)}|v-v_{B_\rho(x)}|^q \, d\mu\r)^{1/q} \le c_P \rho\left(\vint_{ B_{\lambda\rho}(x)} g_v^p \, d\mu\right)^{1/p},
\end{equation}
for every $v\in N^{1,p}(X)$ and $B_\rho(x)\subset X$. In case $\lambda=1$, we say a $(q,p)$-Poincaré inequality is in force. In a general metric measure space setting, it is of interest to have assumptions which are invariant under bi-Lipschitz mappings. The weak $(q,p)$-Poincaré inequality has this quality. 

For a metric space $X$ equipped with a doubling measure $\mu$, it is a result by Hajlasz and Koskela \cite{HajlKosk95} that the following Sobolev inequality holds: If  $X$ supports a weak $(1,p)$-Poincaré inequality for some $1<p<\infty$, then $X$ also supports a weak $(\kappa,p)$-Poincaré
inequality, where
\begin{equation*}%\label{sobolev_exponent}
\kappa=\begin{cases}\frac{d_\mu p}{d_\mu-p},  &\text{for} \quad 1<p<d_\mu, \\ 2p, &\text{otherwise}, \end{cases}
\end{equation*}
possibly with different constants $c_P'>0$ and $\lambda'\geq 1$. %As a consequence of this, by the $(\kappa,p)$-Poincaré inequality and for example by Poposition 5.41 in \cite{BjorBjor11}, there exists a positive constant $C$  such that 
%\begin{align}\label{Sobolev}
%\left(\vint_{B_\rho(x)} |v|^\kappa \, d\mu\right)^{1/\kappa}\le Cr\left(\vint_{B_{\rho}(x)} g_v^p\, d\mu\right)^{1/p}.
%\end{align}
%for every $v\in N_0^{1,p}(X)$ and $B(x,r)\subset X$.

\begin{remark} \label{poincare_remark}
It is a recent result by Keith and Zhong \cite{KeitZhon08}, that when $1<p<\infty$ and $(X,d)$ is a complete metric space with doubling measure $\mu$, the weak $(1,p)$-Poincaré inequality implies a weak $(1,q)$-Poincar\'e inequality for some $1<q<p$. Then by the above discussion, $X$ also supports a weak $(\kappa,q)$-Poincaré inequality with $\kappa>q$ as above. By Hölder's inquality, we can assume that $q$ is close enough to $p$, so that $\kappa\geq p$. By Hölder's inequality, the left hand side of the weak $(\kappa,q)$-Poincaré inequality can be estimated from below by replacing $\kappa$ with any positive $\kappa'<\kappa$. Hence we conclude, that if $X$ supports a weak $(1,p)$-Poincar\'e inequality with $1<p<\infty$, then $X$ also supports a weak $(p,q)$-Poincar\'e and a weak $(q,q)$-Poincar\'e inequality with some $1<q<p$.
\end{remark}

%If $X$  supports a weak $(1,p)$-Poincar\'e inequality and $\mu$ is
%doubling, it follows that Lipschitz functions are dense in
%$\N(X)$.
%We end this
%section by recalling that $f_\limplus=\max\{f,0\}$
% and $f_\limminus=\max\{-f,0\}$.

\subsection{Parabolic spaces and upper gradients}

For $1<p<\infty$, we say that
\[
u\in L^p(0,T;N^{1,p}(\Om)),
\]
if the function $x
\mapsto u(x,t)$ belongs to $N^{1,p}(\Om)$ for almost every $0 < t <
T$, and $u(x,t)$ is measurable as a mapping from $(0,T)$ to $N^{1,p}(\Om)$, that is, the preimage on $(0,T)$ for any given open set in $N^{1,p}(\Om)$ is measurable. %{\bf we could also assume that $\norm{u}_{\N}$ is measurable, but then the density proof would become longer},
Furthermore, we require that the norm
\[
\|u\|_{L^p(0,T;N^{1,p}(\Omega))}=\left(\int_{0}^{T}\norm{u}^p_{N^{1,p}(\Om)}dt\right)^{1/p}
\]
is finite. Analogously, we define 
$L^p(0,T;N_0^{1,p}(\Om))$ and $L_{\textrm{loc}}^p(0,T;N_{\textrm{loc}}^{1,p}(\Omega))$. The space of compactly supported Lipschitz-continuous functions $\Lip_c(\Om_T)$ consists of functions $u,\, \textrm{supp}\, u\subset \Om_T$, for which there exists a positive constant $C_{\Lip}(u)$ such that
\[
\begin{split}
\abs{u(x,t)-u(y,s)}\leq C_{\Lip}(u) (d(x,y)+\abs{t-s}),
\end{split}
\]
whenever $(x,t),(y,s) \in \Om_T$.
The parabolic minimal $p$-weak upper gradient of a function $u\in L_{\textrm{loc}}^p(t_1,t_2;N_{\textrm{loc}}^{1,p}(\Omega))$ is defined in a natural way by setting
\begin{align*}
g_u(x,t)=g_{u(\cdot,t)}(x), %\quad (x,t)\in\Omega\times (0,T),
\end{align*}
at $\nu$-almost every $(x,t)\in \Omega\times (0,T)$. When $u$ depends on time, we refer to $g_u$ as the upper gradient of $u$. The next Lemma on taking limits of upper gradients will be used later in this paper. Here and throughout this paper we denote the time wise mollification of a function by
\begin{align*}
f_\varepsilon(x,t)=\int_{-\varepsilon}^\varepsilon \zeta_\varepsilon(s)f(x,t-s)\,ds,
\end{align*}
where $\zeta_\varepsilon$ is  the standard mollifier with support in $(-\varepsilon,\varepsilon)$.
\begin{lemma}\label{convergence of upper gradient} Let $u\in L_\textrm{loc}^p(0,T;N_{\textrm{loc}}^{1,p}(\Omega))$, where $1<p<\infty$. Then the following statements hold:
\begin{itemize} \item[(a)] As $s\rightarrow 0$, we have $g_{u(x,t-s)-u(x,t)}\rightarrow 0$ in $L_{\textrm{loc}}^p(\Omega_T)$.
\item[(b)] As $\varepsilon\rightarrow 0$, we have $g_{u_\varepsilon -u} \rightarrow 0$ pointwise $\nu$-almost everywhere in $\Omega_T$ and in $L_{\textrm{loc}}^p(\Omega_T)$.
\end{itemize}
\begin{proof}
See Lemma 6.8 in \cite{MassSilj11}. 
\end{proof}
\end{lemma}

%We shall use a fact that $\phi \in \Lip_c(\Om_T)$ is differentiable almost everywhere and $\phi_t\in L^1$ by the fundamental theorem of calculus.

\subsection{Parabolic quasiminimizers} \label{section quasiminimizers}
\begin{definition}
Let $\Omega$ be an  open subset of $X$, $u:
\Omega\times (0,T)\to \R$ and $K'\geq 1$. A function $u$
belonging to the parabolic space
$L^2_{\trm{loc}}(0,T;N_{\trm{loc}}^{1,2}(\Omega))$ is a
\emph{parabolic quasiminimizer} if
\begin{equation*}
\begin{split}
\int_{\{\phi\neq 0\}}u\parts{\phi}{t} \,d \nu + \int_{\{ \phi\neq 0\}} E(u) \,d \nu\leq K' \int_{\{ \phi\neq 0\}} E(u+\phi)\,d \nu,
\end{split}
\end{equation*}
for every $\phi \in \textrm{Lip}(\Omega_T)$  such that $\{\phi\neq 0\}\subset \subset \Omega_T$, where
we denote
$E(u)=F(x,t,g)$ and $F: \Omega \times (0,T) \times \R \to \R $ satisfies the following assumptions:
\begin{enumerate}
\item $ (x,t)\mapsto F(x,t,\xi)$ is
  measurable for every $\xi$,
\item $\xi \mapsto F(x,t,\xi)$ is continuous for almost every $(x,t)$,
\item  there exist $0<c_1\leq c_2<\infty$ such that for every
  $\xi$ and almost every $(x,t)$, we have
\begin{equation*}
c_1 \abs{\xi}^2 \leq F(x,t,\xi)\leq c_2
\abs{\xi}^2.\end{equation*}
\end{enumerate}
\end{definition}
As a consequence of the above, a parabolic quasiminimizer $u$ satisfies
\begin{align}\label{qmin}
\alpha \int_{\{\phi\neq 0\}} u \frac{\partial \phi}{\partial t}\,d\nu+\int_{\{\phi\neq 0\}} g_u^2\,d\nu \leq K \int_{\{\phi\neq 0\}} g_{u+\phi}^2\,d\nu,
\end{align}
with $K=c_2c_1^{-1}K'\geq 1$ and $\alpha=c_1^{-1}$, for every $\phi \in \textrm{Lip}(\Omega_T)$  such that $\{\phi\neq 0\}\subset \subset \Omega$. There is a subtle difficulty in proving Caccioppoli
  type estimates
in the parabolic case: one often needs a test function
depending on $u$ itself, but $u$ is a priori not necessarily in $\Lip(\Om_T)$ nor has compact support. We treat this difficulty
in the following manner. Consider a test function $\phi \in \Lip(\Omega_T)$ with compact support. By a change of variable in \eqref{qmin}, we see that there exists a constant $\varepsilon>0$ such that for every $-\varepsilon<s<\varepsilon$,
\begin{align*}
\alpha\int_{\{\phi\neq 0\}} u(x,t-s) \frac{\partial \phi}{\partial t}\,d\nu+\int_{\{\phi\neq 0\}} g_{u(x,t-s)}^2\,d\nu 
\leq K \int_{\{\phi\neq 0\}} g_{u(x,t-s)+\phi}^2\,d\nu.
\end{align*}
Let now $\zeta_{\eps}(s)$ be a standard
mollifier whose support is contained in $(-\eps,\eps)$. We multiply the above inequality with $\zeta_{\eps}(s)$ and integrate on both sides with respect to $s$ use Fubini's theorem to change the order of integration, and lastly use partial integration for the first term on the left hand side, to obtain
\begin{equation}\label{quasiminimizer inequality}
\begin{split}
-\alpha\int_{\{\phi\neq 0\}} \frac{\partial u_\varepsilon}{\partial t} \phi \,d\nu+\int_{\{\phi\neq 0\}}\left(g_{u}^2\right)_\eps \,d\nu 
\leq K \int_{\{\phi\neq 0\}} \left( g_{u(x,t-s)+\phi}^2\right)_\eps\,d\nu,
\end{split}
\end{equation}
for every compactly supported $\phi\in \Lip(\Omega_T)$. By Lemma 2.3 in \cite{MassMiraParoParv12} we know the following density result: for every $\phi \in L^2(0,T;N^{1,2}(\Omega))$ and $\varepsilon>0$ there exists a function $\varphi \in \Lip(\Omega_T)$ such that $\{\varphi\neq 0\}\subset \subset \Omega_T$ and
\begin{align*}
&\|\phi-\varphi\|_{L^2(0,T;N^{1,2}(\Omega))}<\varepsilon \textrm{ and }\quad \nu(\{\varphi\neq 0\}\setminus \{\phi \neq 0\})<\varepsilon.
\end{align*}
From this it follows, see the proof of Lemma 2.7 in \cite{MassMiraParoParv12}, that if $u\in L_{\textrm{loc}}^2(0,T; N^{1,2}(\Omega))$ is a $K$-quasiminimizer, then \eqref{quasiminimizer inequality} holds for every $\phi\in L_\textrm{c}^{2}(0,T;N_0^{1,2}(\Omega))$. 

 \subsection{Standing assumptions} \label{standing assumptions} We assume that the domain $\Omega$ is regular in the sense that $X\setminus \Omega$ is uniformly $2$-thick. For the definition of thickness see below. Let $\eta:[0,T)\times \Omega\mapsto \R$ be such that $\eta \in W^{1,2}(0,T; N^{1,2}(\Omega))$, $\eta(x,0)\in N^{1,2}(\Omega)$ and
\begin{align*}
\frac{1}{h}\int_0^h \int_\Omega |\eta(x,t)-
 \eta(x,0)|^2\,d\mu \,dt\rightarrow \,0 \textrm{ as }h\rightarrow 0.
\end{align*}
From now on in this paper, we assume that $u\in L^2_{\trm{loc}}(0,T;\N(\Omega))$ is a parabolic quasiminimizer in $\Omega_T$, and satisfies a parabolic boundary condition with $\eta$, in the sense that
 \begin{align}
 \label{lateral boundary condition} & u(\cdot, t)-\eta(\cdot,t)\in N_0^{1,2}(\Omega),\textrm{ for a.e. }t\in(0,T),\\
\label{initial condition 2}& \frac{1}{h} \int_0^h\int_{\Omega}|u(x,t)-\eta(x,t)|^2 \,d\mu\,dt \,\rightarrow 0, \textrm{ as }h\,\rightarrow\,0.
 \end{align}
 
\section{Estimates away from the lateral boundary}\label{away from the lateral boundary}

Establishing higher integrability for a function is based on obtaining a reverse Hölder inequality for the function, and then using it together with a Caldéron--Zygmund type decomposition and a Vitali covering to obtain integrability at some slightly higher exponent. The starting point for showing the reverse Hölder inequality for a parabolic quasiminimizer is an energy estimate over two concentric parabolic cylinders with different radii, $Q_\rho(z_0)$ and $Q_\sigma(z_0)$, where $\rho<\sigma$. This energy estimate is extracted from the definition of parabolic quasiminimizers by choosing a suitable test function.

%When choosing a test function, we want its upper gradient to be of controlled magnitude. In the simplest case, this is obtained by taking advantage of the geometry of two concentric cylinders contained in $\Omega_T$. Roughly speaking, the test function is set to have value $1$ in the smaller cylinder and to have value $0$ outside the larger cylinder, in such a way that difference between the two radii controls the magnitude of the test function's upper gradient.
%
%This purely geometric approach is however hindered if the larger cylinder overlaps the boundary of $\Omega_T$.
%Indeed, the definition of parabolic quasiminimizer requires the test function to have its support compactly contained in $\Omega_T$, and so in this case the effect of the boundary must somehow also be taken into account in the test function.

When choosing the test function, we are faced with two qualitatively different situations. Depending on the center point and radii of the concentric cylinders, the  larger cylinder $Q_\sigma(z_0)=B_\sigma(x_0)\times \Lambda_\sigma(t_0)$, may or may not overlap the lateral boundary of $\Omega_T$. These two alternatives cause a difference in how we build the test function, and consequently lead to different energy estimates.

In case we assume $B_\sigma(x_0)$ is a subset of $\Omega$, and so $Q_\sigma(z_0)$ does not overlap the lateral boundary of $\Omega_T$, we can construct the test function by using only the geometry of the cylinders $Q_\rho(x_0)$ and $Q_\sigma(x_0)$, the quasiminimizer $u$ itself, and the given initial condition, without having to take into consideration the lateral boundary of $\Omega_T$. 

%On the other hand, in case the metric ball of the larger cylinder overlaps the boundary of $\Omega$, we cannot build the test without taking  into account the boundary. Indeed, the definition of parabolic quasiminimizers requires that the test function to vanish on the boundary. To achieve this we make use of given lateral boundary conditions.

We begin by treating this case. In order to establish the reverse Hölder inequality, it turns out  that we only need the energy estimate for $\rho<\sigma\leq 2\rho$. Therefore the discussion in this section covers those cylinders $Q_\rho(x_0)$ for which we have $B_{2\rho}(x_0)\subset \Omega$. The complementary case to this covers the cylinders $Q_\rho(x_0)$ such that $B_{2\rho}(x_0)\setminus \Omega\neq \emptyset$, and is the topic of Section \ref{Estimates near the lateral boundary}.

\begin{lemma}[Energy estimate]
\label{energy estimate near the initial boundary}
There exists a positive constant $
c=c(K)$, such that for every $Q_{\rho}=B_\rho(x_0) \times \Lambda_\rho(t_0)$, $\rho<\sigma$ where $B_\sigma \subset \Omega$, we have
\[
%\label{esicaccpohjaQ}
\begin{split}
 &\esssup_{t \in
  \Lambda_{\rho}\cap (0,T)} \int_{B_{\rho}} |u-u_{\sigma}(t)|^2 \,d \mu +\int_{Q_{\rho}\cap \Omega_T}  g_u^2 \,d \nu \\
&\leq
c \int_{(\q{\sigma}\setminus \q{\rho})\cap \Omega_T} g_u^2 \,d \nu
+\frac{c}{(\sigma-\rho)^2}\frac{\mu(B_\sigma)}{\mu(B_\rho)}\int_{\q{\sigma}\cap \Omega_T}|u-u_{\sigma}(t)|^2
\,d \nu\\
&\qquad+c\frac{\mu(B_\sigma)}{\mu(B_\rho)}\int_{B_\sigma}|\eta(x,0)-\eta_\sigma(0)|^2\,d\mu
\end{split}
\]

%%%%%%%%%%%%%%%%%%%%%%%%%%%%%%%%%%%%%%%%%%%%%%%%%%%%%%%%%%%%%%%%%%%%%%%%%%%%%%%%%%%%%%%%%%%%%%%%%%%%%%%%%%%%%%%%%%%%%%%%
\begin{proof}
Assume $Q_\rho=B_\rho(x_0)\times \Lambda_\rho(t_0)$, and $\rho<\sigma$ are such that $B_\sigma(x_0) \subset \Omega$ $\Lambda_\rho \cap (0,T) \neq \emptyset$. Assume $t'\in \Lambda_\rho \cap (0,T)$. Define
 \begin{align*}
 %&\chi_h(t)=
 %\begin{cases}
 %1, & h\leq t \leq t'-h,\\
  %0, &  t \leq h/2 \textrm{ or }t\geq t'-h/2,
 %\end{cases}\\
%&|\chi_h'(t)| \leq \frac{2}{h} \textrm{ for every }t\in(0,T),\\
 &\chi_h(t) = 
 \begin{cases}
 \frac{t-h}{h}, &h\leq t \leq 2h, \\
 1,& 2h \leq t \leq t'-2h,\\
 \frac{t'-h-t}{h},&t'-2h\leq t \leq t'-h,\\
 0,&\textrm{otherwise}.
\end{cases}
% \textrm{ in }W^{1,2}(0,T) \textrm{ as } h\rightarrow 0.
 \end{align*}
Let $\varphi\in N^{1,2}(B_\sigma)$, $0\leq \varphi\leq 1$, be such that  $\varphi=1$ in $B_\rho$, the support of $\varphi$ is a compact subset of  $B_\sigma$, and 
\begin{align*}
g_\varphi^2\leq \frac{c}{(\sigma-\rho)^2}. 
\end{align*}
For a function $f(x,t)$, denote
\begin{align}\label{weighted average}
f_{\sigma}^\varphi(t)=\frac{\int_{B_\sigma} f(x,t) \varphi(x)\, d\mu}{\int_{B_\sigma} \varphi(x) \, d\mu}.
\end{align}
Set now $\phi=-\varphi(u_{\eps}-(u_\varepsilon)_\sigma^\varphi)\chi_{h}$. Since $u\in L^2_{\trm{loc}}(0,T;\N(\Omega))$ is a parabolic quasiminimizer in $\Omega_T$ and $\phi \in \Lip_c(0,T;N^{1,2}(\Omega))$, by the discussion in Section \ref{section quasiminimizers} we can insert $\phi$ into inequality \eqref{quasiminimizer inequality} and examine the resulting terms. In the first term on the left hand side, we add and subtract
$(u_\eps)_{\sigma}^\varphi$ to obtain after integrating by parts
\begin{equation}\label{valivaihe}
\begin{split}
&-\int_{\Om_T} \frac{\partial u_{\eps}}{\partial t}  \phi\,d \nu=\int_{\Om_T}
(u_{\eps}-(u_\varepsilon)_\sigma^\varphi(t))\parts{\phi}{t}\,d \nu-\int_{\Om_T}\frac{\partial (u_\varepsilon)_\sigma^\varphi(t)}{\partial t}\phi \,d \nu.
\end{split}
\end{equation}
Using the definition of
$(u_\varepsilon)_\sigma^\varphi(t)$, we
  see that the last term on the right hand side vanishes
\[
 \begin{split}
&\int_{\Om_T}\parts{(u_\varepsilon)_\sigma^\varphi(t) }{t}\phi \,d \nu\\
&=
 -\int_{0}^{t'} \parts{ (u^{\varphi}_{\sigma,\eps}(t))}{t} \l(\int_{B_{\sigma}}u_{\eps} \varphi \,d \mu - \frac{\int_{B_{\sigma}}
 \varphi\,d \mu \int_{B_{\sigma}} u_{\eps} \varphi \,d
 \mu}{\int_{B_{\sigma}}\varphi \,d \mu}\r)\chi_{h}(t) \,d t
=0.
 \end{split}
\]
Obtaining this vanishing property is one of the two reasons for defining the weighted average \eqref{weighted average}. The other reason is that the integral average of the function $|u-u_ \sigma^\varphi|^2$ over $B_\sigma$ is comparable to the integral average of $|u-u_\sigma|^2$ over $B_\sigma$, as will be seen at the end of the proof. We write out the first term on the right hand side of \eqref{valivaihe},
and have
\[
\begin{split}
-\int_{\Om_T} \parts{u_{\eps}}{t}  \phi \,d \nu&=-\int_{\Om_T}
(u_{\eps}-(u_\varepsilon)_\sigma^\varphi(t))^2\parts{}{t}\left(\varphi\chi_{h}\right) \,d \nu
\\
&\hspace{1 em}-\frac{1}{2}\int_{\Om_T}
\parts{}{t}\left((u_{\eps}-(u_\varepsilon)_\sigma^\varphi(t))^2\right)(\varphi\chi_{h}) \,d \nu\\
&=-\frac{1}{2}\int_{\Om_T}
(u_{\eps}-(u_\varepsilon)_\sigma^\varphi(t))^2\parts{}{t}\left(\varphi\chi_{h}\right) \,d \nu,
\end{split}
\]
and so, taking into account the definition of $\chi_h$, we arrive at
\begin{equation}\label{aikakE}
\begin{split}
 -\int_{\Om_T} &\parts{u_{\eps}}{t} \phi \,d \nu=-\frac{1}{2h} \int_h^{2h}
\int_{B_{\sigma}}\abs{u_\eps(x,t)-(u_\eps)_{\sigma}^{\varphi}(t)}^2
\varphi(x) \,d \mu\\
&\qquad+\frac{1}{2h} \int_{t'-2h}^{t'-h}
\int_{B_{\sigma}}\abs{u_\eps(x,t)-(u_\eps)_{\sigma}^{\varphi}(t)}^2
\varphi(x) \,d \mu.
\end{split}
\end{equation}
By the definition \eqref{weighted average}, we have for every $\varepsilon<h$
\begin{align*}
\int_{\{\phi\neq 0\}}(u_\sigma^\varphi-(u_\eps)_\sigma^\varphi)^2\, d\nu\leq \mu(\Omega)\int_{h-\delta}^{T-h+\delta}(u_\delta^\varphi(t)-(u_\sigma^\varphi)_\eps(t))^2\,dt\\
\leq \mu(\Omega)\int_{-\varepsilon}^{\varepsilon} \int_{h-\delta}^{T-h+\delta}|u_\sigma^\varphi(t)-u_\sigma^\varphi(t-s)|^2\,dt \, \zeta_{\eps}(s) \,ds,
\end{align*}
and therefore the fact that $u_{\sigma}^\varphi\in L_\textrm{loc}^2(0,T)$ implies that the above expression tends to zero as $\eps \rightarrow 0$.
Hence, using the triangle inequality and the initial condition \eqref{initial condition 2} as first $\eps\to 0$ and then $h\to 0$ leads us to
\begin{align*}
\lim_{\varepsilon,h\rightarrow 0} -\int_{\Om_T} \parts{u_{\eps}}{t} \phi \,d \nu=&-\frac{1}{2}
\int_{B_{\sigma}}\abs{\eta(x,0)-\eta_{\sigma}^{\varphi}(0)}^2
\varphi(x) \,d \mu\\
&+\frac{1}{2} 
\int_{B_{\sigma}}\abs{u(x,t')-u_{\sigma}^{\varphi}(t')}^2
\varphi(x) \,d \mu.
\end{align*}
On the right hand side of inequality \eqref{quasiminimizer inequality}, we note that for every $h,\eps$, in the set $\{\phi\neq0\}$ we have
\begin{align*}
&(g_{u(\cdot,\cdot-s)-\varphi(u_\eps-(u_\eps)_\sigma^\varphi)\chi_{h}}^2)_\eps\leq c (g^2_{u(\cdot,\cdot-s)-u})_\eps+ cg^2_{u-\varphi(u-u_\sigma^\varphi)}\\&\qquad+cg^2_{\varphi(u-u_\sigma^\varphi)}(1-\chi_{h})^2
+cg_\varphi^2((u_\eps)_\sigma^\varphi-u_\sigma^\varphi)^2\chi_{h}^2\\&\qquad+c(u-u_\eps)^2 g_{\varphi}^2 \chi_{h}^2+c\varphi^2 g_{u-u_\eps}^2 \chi_{h}^2.
\end{align*}
By Lemma \ref{convergence of upper gradient} we know that $g_{u-u_\eps}^2\rightarrow 0$  and $(g_{u(\cdot,\cdot-s)-u}^2)_\eps\rightarrow 0$ in $L_{\textrm{loc}}^1(\Omega_T)$ as $\eps\rightarrow 0$. Hence, we obtain
\begin{align*}
\limsup_{\varepsilon,  h\rightarrow 0}\int_{\{\phi \neq 0\}} \left(g_{u(\cdot,\cdot-s)+\phi}^2 \right)_\varepsilon\,d\nu \leq c\int_{Q_\sigma \cap \Omega_T}g_{u-\varphi(u-u_\sigma^\varphi)}^2 \, d\mu\, dt.
\end{align*}
Now we note that since $u_\sigma^\varphi$ does not depend on $x$ and hence its upper gradient vanishes, we have
\begin{align*}
\int_{Q_\sigma \cap \Omega_T}g_{u-\varphi(u-u_\sigma^\varphi)}^2 \, d\nu \leq \int_{Q_\sigma \cap \Omega_T}|1-\varphi|^2g_{u}^2 \, d\nu +\int_{Q_\sigma \cap \Omega_T}|u-u_\sigma^\varphi(t)|^2g_{\varphi}^2 \, d\nu.
\end{align*}
Combining the obtained expressions leads us to the estimate
\[
%\label{esicaccpohjaQ}
\begin{split}
 &\esssup_{t \in
  \Lambda_{\rho}\cap (0,T)} \int_{B_{\rho}} |u-u_{\sigma}^{\varphi}(t)|^2 \,d \mu +\int_{Q_{\rho}\cap \Omega_T}  g_u^2 \,d \nu \\
&\quad\leq
c \int_{(\q{\sigma}\setminus \q{\rho})\cap \Omega_T} g_u^2 \,d \nu
+\frac{c}{(\sigma-\rho)^2}\int_{\q{\sigma}\cap \Omega_T}|u-u_{\sigma}^\varphi(t)|^2
\,d \nu\\
&\qquad+c
\int_{B_{\sigma}}\abs{\eta(x,0)-\eta_{\sigma}^{\varphi}(0)}^2
\,d \mu,
\end{split}
\]
where $c=c(K)$. We complete the  proof by noting that for any $t\in (0,T)$, we have
\begin{align*}
\int_{B_\rho}&|u-u_{\sigma}(t)|^2\,d\mu\\
 %&\leq 2\int_{B_\rho}|u(x,t)-u_{\sigma}^\varphi(t)|^2\,d\mu+2\int_{B_\rho}|u_{\sigma}^\varphi(t)-u_{\sigma}(t)|^2\,d\mu\\
 &\leq 2\int_{B_\rho}|u-u_{\sigma}^\varphi(t)|^2\,d\mu+2\int_{B_\rho}\left(\vint_{B_{\sigma}}|u_{\sigma}^\varphi(t)-u|^2\,d\mu \right)\,d\mu\\
 &\leq  4\int_{B_\rho}|u-u_{\sigma}^\varphi(t)|^2\,d\mu.
 \end{align*}
On the other hand, by the triangle inequality and by Jensen's inequality, and since $\varphi=1$ in $B_\rho$,
 \begin{align*}
 \int_{B_{\sigma}}&|u-u_{\sigma}^\varphi(t)|^2 \, d\mu
% \leq c\int_{Q_{\sigma}\cap \Omega_T}|u-u_{\sigma}(t)|^2 \, d\nu\\+c\int_{Q_{\sigma}\cap \Omega_T}|u_{\sigma}(t)-u_{\sigma}^\varphi(t)|^2 \, d\nu
\leq 2\int_{B_{\sigma}}|u-u_{\sigma}(t)|^2 \, d\mu\\&+2\int_{B_{\sigma}}\left( \left(\int_{B_{\sigma}}\varphi\,d\mu\right)^{-1}\int_{B_{\sigma}}|u_{\sigma}(t)-u|^2\varphi \,d\mu\right) \, d\mu\\
 &\leq 4\frac{\mu(B_{\sigma})}{\mu(B_\rho)} \int_{B_{\sigma}}|u-u_{\sigma}(t)|^2 \, d\mu.
 \end{align*}
 The analogous applies for the functions $\eta(x,0)$, $\eta_\sigma(0)$ and $\eta_\sigma^\varphi(0)$.
\end{proof}
\end{lemma}

Having established the fundamental energy estimate, we derive a Caccioppoli inequality  by using the hole filling
iteration. For the iteration, it is essential that we can write the above energy estimate for every $\sigma\in(\rho,2\rho)$.
\begin{lemma}[Caccioppoli]
\label{caccioppoli near the initial boundary}
There exists a positive constant $c=c(c_\mu,K)$ so that for any $Q_\rho=B_\rho(x_0)\times \Lambda_\rho(t_0)$, such that $B_{2\rho}(x_0)\subset \Omega$, we have
\[
\label{eq:cacckvasi}
\begin{split}
\int_{Q_{\rho}\cap\Omega_T} g_u^2 \, d \nu &\leq
\frac{c}{\rho^2}\int_{\q{2\rho}\cap \Omega_T}|u-u_{2\rho}(t)|^2  \, d \nu+c\int_{B_{2\rho}}|\eta(x,0)-\eta_{2\rho}(0)|^2\,d\mu.
\end{split}
\]

%%%%%%%%%%%%%%%%%%%%%%%%%%%%%%%%%%%%%%%%%%%%%%%%%%%%%%%%%%%%%%%%%%%%%%%%%%%%%%%%%%%%%%%%%%%%%%%%%%%%%%%%%%%%%%%%%%%%%%%%
\begin{proof} By Lemma \ref{energy estimate near the initial boundary}, for any cylinder $Q_\rho=B_\rho(x_0)\times \Lambda_\rho(t_0)$ such that $B_{2\rho}(x_0)\subset \Omega$, we have for any $\rho<\sigma\leq 2\rho$,
\[
%\label{esicaccpohjaQ}
\begin{split}
 &\esssup_{t \in
  \Lambda_{\rho}\cap (0,T)} \int_{B_{\rho}} |u-u_{\sigma}(t)|^2 \,d \mu +\int_{Q_{\rho}\cap \Omega_T}  g_u^2 \,d \nu \\
&\leq
c \int_{(\q{\sigma}\setminus \q{\rho})\cap \Omega_T} g_u^2 \,d \nu
+\frac{c}{(\sigma-\rho)^2}\frac{\mu(B_{\sigma})}{\mu(B_\rho)}\int_{\q{\sigma}\cap \Omega_T}|u-u_{\sigma}(t)|^2
\,d \nu\\
&\qquad+c\frac{\mu(B_\sigma)}{\mu(B_\rho)}\int_{B_\sigma}|\eta(x,0)-\eta_\sigma(x,0)|^2\,d\mu,
\end{split}
\]
where $c=c(K)$. We add $c\int_{Q_\rho} g_u^2\,d\nu$ to both sides of the expression, and divide by $1+c$, to obtain
 \begin{align*}
 &\int_{Q_\rho} g_u^2 \, d\nu\leq  \frac{c}{1+c} \int_{Q_{\sigma}} g_u^2 \,d\nu+\frac{c}{(1+c)(\sigma-\rho)^2}\frac{\mu(B_{\sigma})}{\mu(B_\rho)}\int_{Q_{\sigma}}|u-u_{\sigma}(t)|^2
\, d \nu\\
&\qquad+\frac{c}{(1+c)}\frac{\mu(B_{\sigma})}{\mu(B_\rho)}\int_{B_{\sigma}}|\eta(x,0)-\eta_{\sigma}(x,0)|^2\,d\mu.
 \end{align*}
Then we choose 
 \begin{align*}
 \rho_0=\rho,\quad \rho_i-\rho_{i-1}=\frac{1-\beta}{\beta}\beta^i \rho, \quad i=1,2,\dots,k,\quad \beta^2=\frac{1}{2}\left(\frac{c}{1+c}+1\right),
 \end{align*}
 replace $\rho$ by $\rho_{i-1}$ and $\sigma$ by $\rho_i$, and iterate, to have
 \begin{align*}
 \int_{Q_\rho} g_u^2 \, d\nu&\leq  \left(\frac{c}{1+c}\right)^k \int_{Q_{\rho_k}} g_u^2 \,d\nu\\
 &+\sum_{i=1}^{k}\left(\frac{c}{1+c}\right)^i\frac{\mu(B_{\rho_i})}{\mu(B_{\rho_{i-1}})}\left(\frac{1}{(\rho_i-\rho_{i-1})^2}\int_{Q_{\rho_i}}|u-u_{\rho_i}(t)|^2
\, d \nu\right.\\
&\qquad\left.+\int_{B_{\rho_i}}|\eta(x,0)-\eta_{\rho_i}(0)|^2\,d\mu\right).
 \end{align*}
 Here among other things $\rho_i\leq 2\rho_{i-1}$ for every $i$, and so by the doubling property  of $\mu$, the ratio $\mu(B_{\rho_i})/\mu(B_{\rho_{i-1}})$ is uniformly bounded. Also, for each $i$ we can estimate  after using Fubini's theorem,
 \begin{align*}
 \int_{Q_{\rho_i}} &|u-u_{\rho_i}(t)|^2\,d\nu\leq 2 \int_{Q_{2\rho}} |u-u_{2\rho}(t)|^2\,d\nu\\
& + 2 \int_{Q_{2\rho}} \vint_{B_{\rho_i}}|u_{2\rho}(t)-u|^2\,d\mu\,d\nu\leq 2c \int_{Q_{2\rho}}|u-u_{2\rho}(t)|^2\,d\nu,
 \end{align*}
where $c=c(c_\mu)$, and analogously for $\eta, \eta_{\rho_i}$. Hence, taking the limit $k\rightarrow \infty$ yields the estimate,
 \begin{align*}
\int_{Q_\rho} g_u^2 \, d\nu\leq \frac{c}{\rho^2}\int_{Q_{2\rho}} |u-u_{2\rho}(t)|^2
\, d \nu+c\int_{B_{2\rho}}|\eta(x,0)-\eta_{2\rho}(0)|^2\,d\mu, 
 \end{align*} 
 where $c=c(c_\mu,K)$.
\end{proof}
\end{lemma}
 
Next we prove a parabolic version of the Poincaré inequality. We use the fundamental energy estimate with $\sigma=2\rho$.
 \begin{lemma}[Parabolic Poincaré] \label{parabolic Poincaré near the initial boundary}
 There exists positive constants\\ $c=c(c_\mu,c_P,\lambda, K)$  and $1<q_0<2$ so that for any $Q_\rho=B_\rho(x_0)\times \Lambda_\rho(t_0)$, such that $B_{2\rho}(x_0)\subset \Omega$, we have
\[
%\label{esicaccpohjaQ}
\begin{split}
\esssup_{t \in
  \Lambda_{\rho}\cap (0,T)} \vint_{B_{\rho}} |u-u_{\rho}(t)|^2 \, d \mu
\leq
c &\frac{\rho^2}{\nu(Q_{2\lambda\rho})}\int_{\q{2\lambda\rho}\cap \Omega_T} g_u^2 \,d \nu\\&\qquad+ c\rho^2\left( \vint_{B_{2\lambda\rho}} g_\eta^q(x,0)\,d\mu \right)^\frac{2}{q},
\end{split}
\]
for any $q_0\leq q$.
%%%%%%%%%%%%%%%%%%%%%%%%%%%%%%%%%%%%%%%%%%%%%%%%%%%%%%%%%%%%%%%%%%%%%%%%%%%%%%%%%%%%%%%%%%%%%%%%%%%%%%%%%%%%%%%%%%%%%%%%
\begin{proof}

 By Lemma \ref{energy estimate near the initial boundary}
 \[
%\label{esicaccpohjaQ}
\begin{split}
 &\esssup_{t \in
  \Lambda_{\rho}\cap (0,T)} \int_{B_{\rho}} |u-u_{2\rho}|^2 \,d \mu \leq
c \int_{\q{2\rho}\cap \Omega_T} g_u^2 \,d \nu\\
&\qquad+\frac{c}{\rho^2}\int_{\q{2\rho}\cap \Omega_T}|u-u_{2\rho}|^2
\,d \nu+c\int_{B_{2\rho}}|\eta(x,0)-\eta_{2\rho}(0)|^2\,d\mu.
\end{split}
\]
Since by assumption $B_{2\rho}\subset \Omega$, we can use the $(2,2)$-Poincaré  inequality for the second term on right hand side, and the $(2,q)$-Poincaré inequality, where $1<q<2$ is as in Remark \ref{poincare_remark}, for the third term on the right hand side. We obtain
\begin{align*}
&\esssup_{t \in
  \Lambda_{\rho}\cap (0,T)} \vint_{B_{\rho}} |u-u_{2\rho}|^2 \,d \mu
  % c \int_{0}^{t_0+2\rho^2}\vint_{B_{2\rho}} g_u^2 \,d \mu\,dt\\
  \\
  &\qquad \leq c\int_{\Lambda_\rho \cap (0,T)}\vint_{B_{2\lambda\rho}}g_u^2\,d \mu\,dt+ c\rho^2\left( \vint_{B_{2\lambda\rho}} g_\eta^q(x,0)\,d\mu \right)^\frac{2}{q},
\end{align*}
where $c=c(c_\mu,c_P,K)$. The proof is completed by observing that
$Q_{2\rho}=4\rho^2\nu(B_{2\rho})$.
 \end{proof}
\end{lemma}
Caccioppoli's inequality together with the parabolic- and $(2,q)$-Poincaré inequality now provide us the required tools to establish a reverse Hölder's inequality.
\begin{lemma}[Reverse Hölder inequality]\label{reverse hölder away from boundary}
There exists a positive constant $c=c(c_\mu,c_P,\lambda, K)$, and a $1<q<2$, so that for any $Q_\rho=B_\rho(x_0)\times \Lambda_\rho(t_0)$, such that $B_{2\rho}(x_0)\subset \Omega$ , we have
 \begin{align*}
 &\frac{1}{\nu(Q_\rho)}\int_{Q_{\rho}\cap \Omega_T} g_u^2 \, d \nu\leq \varepsilon c\frac{1}{\nu(Q_{2\rho})}\int_{\q{2\lambda\rho}\cap \Omega_T} g_u^2 \,d \nu\\
 &\qquad+\varepsilon^{-1}c\left(\frac{1}{\nu(Q_{2\rho})}\int_{Q_{2\lambda \rho}\cap \Omega_T} g_u^q\,d\nu\right)^\frac{2}{q}+ \varepsilon c\left( \vint_{B_{2\lambda\rho}} g_\eta^q(x,0)\,d\mu \right)^\frac{2}{q}.
 \end{align*}
 \begin{proof}
 By the Caccioppoli Lemma \ref{caccioppoli near the initial boundary}, by the doubling property of $\mu$ and since $\nu(Q_{\rho})=\rho^2\mu(B_\rho)$, and then the $(2,q)$-Poincaré inequality for the second term on the right hand side, we obtain
 \begin{align*}
\frac{1}{\nu(Q_\rho)}&\int_{Q_{\rho}\cap \Omega_T} g_u^2 \, d \nu\\ %&\leq
%\frac{c}{\rho^2}\frac{1}{\nu(Q_\rho)}\int_{\q{2\rho}\cap\Omega_T}|u-u_{2\rho}(t)|^2  \, d \nu\\
%&\qquad+c\frac{1}{\rho^2}\vint_{B_{2\rho}}|\eta-\eta_{2\rho}|^2\,d\mu\\
&\leq \frac{c}{\rho^4}\int_{\Lambda_\rho\cap \Omega_T}\vint_{B_{2\rho}}|u-u_{2\rho}|^2  \, d \mu\,dt +c\left(\vint_{B_{2\lambda\rho}}g_{\eta}^q(x,0)\,d\mu\right)^\frac{2}{q},
%&\leq  \frac{c}{\rho^4} \left(\esssup_{t \in
 % \Lambda_{\rho}\cap (0,T)} \vint_{B_{\rho}} |u-u_{\rho}(t)|^2 \, d \mu\right)^{1-\frac{q}{2}} \int_0^{t_0+2\rho^2}\left(\vint_{B_{2\rho}}|u-u_{2\rho}(t)|^2  \, d \mu\right)^\frac{q}{2}\,dt
 \end{align*}
 where $c=c(c_\mu,c_P,K)$. On the other hand we can write
 \begin{align*}
 &\frac{c}{\rho^4}\int_{\Lambda_\rho\cap \Omega_T}\vint_{B_{2\rho}}|u-u_{2\rho}|^2  \, d \mu\,dt \leq \left(\frac{c}{\rho^2}\esssup_{t \in
  \Lambda_{\rho}\cap (0,T)} \vint_{B_{\rho}} |u-u_{2\rho}|^2 \, d \mu\right)^{1-\frac{q}{2}}\\
& \qquad\cdot \frac{c}{\rho^2} \int_{\Lambda_\rho\cap \Omega_T}\left(\frac{c}{\rho^2} \vint_{B_{2\rho}}|u-u_{2\rho}|^2  \, d \mu\right)^\frac{q}{2}\,dt\\
&   \leq \left\{\frac{c}{\nu(Q_{2\rho})} \int_{\q{2\lambda\rho}\cap \Omega_T} g_u^2 \,d \nu+c\left( \vint_{B_{2\lambda\rho}} g_\eta^q(x,0)\,d\mu \right)^\frac{2}{q}\right\}^{1-\frac{q}{2}}\\
&\qquad\cdot \frac{c}{\nu(Q_{2\rho})}\int_{Q_{2\lambda \rho}\cap \Omega_T} g_u^q\,d\nu, 
 \end{align*}
 were we used Lemma \ref{parabolic Poincaré near the initial boundary} and the $(2,q)$-Poincaré inequality. By the $\varepsilon$-Young inequality we now obtain for every positive $\varepsilon$
 \begin{align*}
 &\frac{1}{\nu(Q_\rho)}\int_{Q_{\rho}\cap \Omega_T} g_u^2 \, d \nu\leq \varepsilon c\frac{1}{\nu(Q_{2\rho})}\int_{\q{2\lambda\rho}\cap \Omega_T} g_u^2 \,d \nu\\
 &\qquad+\varepsilon^{-1}c\left(\frac{1}{\nu(Q_{2\rho})}\int_{Q_{2\lambda \rho}\cap \Omega_T} g_u^q\,d\nu\right)^\frac{2}{q}+ \varepsilon c\left( \vint_{B_{2\lambda\rho}} g_\eta^q(x,0)\,d\mu \right)^\frac{2}{q},
 \end{align*}
 where $c=c(c_\mu,c_P, \lambda,K)$.
 \end{proof}
\end{lemma}

%%%%%%%%%%%%%%%%%%%%%%%%%%%%%%%%%%%%%%%%%%%%%%%%%%%%%%%%%%%%%%%%%%%%%%%%%%%%%%%%%%%%%%%%%%%%%%%%%%%%%%%%%%%%%%%%%%%%%%%%
 \section{Estimates near the lateral boundary}\label{Estimates near the lateral boundary}

In this section we treat the almost complementary case to the one covered in section \ref{away from the lateral boundary}. This means that we establish a reverse Hölder estimate for parabolic quasiminimizers in the cylinders $Q_\rho(z_0)=B_{\rho}(x_0)\times \Lambda_\rho(t_0)$ which are such that $B_{2\rho}(x_0)\setminus \Omega\neq \emptyset$. However, in addition to this we will have to assume that $\rho$ is small enough, and so we cover the situation where the cylinders are contained in the vicinity of the lateral boundary.

Continuing the discussion from the beginning of Section \ref{away from the lateral boundary}, here in the case were $Q_\sigma(z_0)$ may overlap the lateral boundary of $\Omega_T$, we have to take the lateral boundary of $\Omega_T$ into consideration when building the test function for obtaining the energy estimate.  Indeed, instead of relying solely on the geometry of $Q_\rho(z_0)$ and $Q_\sigma(z_0)$, we also make use of the lateral boundary  condition.

After obtaining the energy estimate, as a consequence of building the lateral boundary condition into the test function, we cannot use the usual Poincaré inequality to the same extent as was done in section \ref{away from the lateral boundary}. Instead we use a version of the Poincaré inequality which introduces the variational capacity of the zero set of the function $u-\eta$.

Before going on, we introduce some concepts.
 
 \begin{definition}Let $F\subset X$ be an open set, and let $E\subset F$. The variational capacity is defined
 \begin{align*}
 \textrm{cap}_p(E,F)=\inf_{f} \int_{F} g_f^p \, d\mu,
 \end{align*}
 where the infimum is taken over all $f\in N_0^{1,p}(F)$ such that $f\geq 1$ on $E$. 
 \end{definition}
 As can be seen from the following, in our setting the variational capacity is closely related to the measure of the sets.
 \begin{lemma}\label{estimate of capacity}Let $X$ be a measure space equipped with a doubling measure $\mu$, and satisfies a weak $p$-Poincaré inequality. Let $E \subset B_\rho(x)$ with $0 < \rho < (1/8)$diam$(X)$. Then there exists a positive constant $c=c(c_P,c_\mu,\lambda, p)$ such that
 \begin{align*}
\frac{\mu(E)}{c \rho^p} \leq \textrm{cap}_p(E, B_{2\rho}(x)) \leq c \frac{\mu(B_\rho (x))}{\rho^p}.
\end{align*}
\begin{proof} For proof we refer the reader to \cite{Bjor02}.
\end{proof}
 \end{lemma}
  We will need the following version of Poincaré's inequality, which gives an upper gradient estimate for the integral average of any Newtonian function. We use the self improving property of the usual Poincaré inequality to obtain $1<q<2$ on the right hand side. 
 \begin{theorem}[Poincaré with capacity]\label{Poincaré with capacity}
Suppose $f\in N^{1,2}(B_{2\rho})$. Denote $N_{B_\rho}(f)=\{\,x\in B_\rho\,:\,f(x)=0\,\}$. Then there exists a $1<q_0<2$ and a positive constant $c=c(c_\mu,c_P,\lambda)$ such that for any $q_0\leq q\leq 2$, we have
 \begin{align*}
 \left( \kint_{B_{2\rho}} |f|^2 \, d\mu \leq \right)^{\frac{1}{2}} \leq c \left( \frac{1}{\textrm{cap}_q(N_{B_\rho}(f), B_{2\rho})} \int_{B_{2\lambda \rho} }g_f^q\, d\mu \right)^{\frac{1}{q}},
 \end{align*}   
 for every $0<\rho<(1/8)$diam$(X)$.
 \begin{proof}
 First we assume that
 \begin{align*}
 f_{B_{2\rho}}=\kint_{B_{2\rho}} f(x)\, d\mu\neq 0.
 \end{align*}
Take $\phi \in$ Lip$_c (B_{2\rho})$ and $0\leq \phi \leq 1$, such that $\phi=1$ in $B_\rho$ and  $g_\phi\leq \frac{2}{\rho}$. Define $v:\, X\rightarrow \R$ by setting
\begin{align*}
v= \begin{cases} \phi (f_{B_{2\rho}}-f )&\mbox{ in } B_{2\rho}\\
0 &\mbox{ in } x\in X\setminus B_{2\rho}. \end{cases}.
\end{align*}
Then $v \in N^{1,2} (X)$, the support of $v$ is a compact subset of $B_{2\rho}$ and we have $v=f_{B_{2\rho}}-f$ in $B_\rho$.
%\begin{align*}
%v \mbox{ is }q\mbox{-quasicontinuous},\\
%\mbox{supp}v \subset C \subset B_{2\rho},\quad C \mbox{ is compact},\\
%.
%\end{align*}
From Remark \ref{poincare_remark} we know there exists a $1<q_0<2$ so that for any $q_0\leq q \leq 2$  the weak $(q,q)$-Poincaré inequality holds.
By the product rule for upper gradients and then the $(q,q)$- Poincaré inequality, %since $u\in N_0^{1,q}(B_{2\rho})$,
\begin{align*}
\int_{B_{2\rho}} g_v^q \,d\mu &\leq \int_{B_{2\rho}} (g_f \phi + |f_{B_{2\rho}}-f|g_\phi)^q\, d\mu \\
&\leq c\int_{B_{2\rho}} g_f^q \, d\mu + c\frac{2^q}{\rho^q} \int_{B_{2\rho}}|f_{B_{2\rho}}-f|^q \, d\mu%\leq c\int_{B_{2\rho}} g_u^q \, d\mu +c2^qc_P \int_{B_{\lambda 2\rho}} g_u^q \, d\mu
\leq c \int_{B_{\lambda 2\rho}} g_f^q \, d\mu,
\end{align*}
where $c=c(c_P)$. On the other hand, since $N_{B_\rho} (f) \subset \{\,f_{B_{2\rho}}^{-1} v=1\,\}$, we have by the definition of the $q$-capacity
\begin{align*}
\frac{1}{|f_{B_{2\rho}}|^{q}}\int_{B_{2\rho}} g_v^q \, d\mu = \int_{B_{2\rho}} g_{f_{B_{2\rho}}^{-1} v}^q \, d\mu\geq \textrm{cap}_q(N_{B_\rho}(f), B_{2\rho}).
\end{align*}
This gives us
\begin{align*}
|f_{B_{2\rho}}|\leq \left( \frac{1}{\textrm{cap}_q(N_{B_\rho}(f), B_{2\rho})} \int_{B_{2\rho}} g_v^q \, d\mu\right)^{\frac{1}{q}}.
\end{align*}
Now we can use the $(2,q)$-Poincaré inequality together with the above inequality, and then Lemma \ref{estimate of capacity} to write
\begin{align*}
&\left( \kint_{B_{2\rho}} |f|^2 \, d\mu \right)^\frac{1}{2} \leq \left( \kint_{B_{2\rho}} |f_{B_{2\rho}}-f|^2 \, d\mu \right)^\frac{1}{2}+|f_{B_{2\rho}}| \\
&\qquad\leq  c_P\rho \left( \kint_{B_{2\lambda \rho}} g_f^q \, d\mu \right)^\frac{1}{q}+\left( \frac{1}{\textrm{cap}_q(N_{B_\rho}(f), B_{2\rho})} \int_{B_{2\rho}} g_v^q \, d\mu\right)^{\frac{1}{q}} \\
%\leq c \lambda \rho \left( \kint_{B_{\lambda 2 \rho}} g_u^q \, d\mu \right)^\frac{1}{q} + \left( \frac{1}{\textrm{cap}_q(N_{B_\rho}(u), B_{2\rho})} \int_{B_{2\rho}} g_v^q \, d\mu\right)^{\frac{1}{q}}\\
%\leq c \lambda \left( \frac{\rho^q}{\mu(B_{2\rho})}\int_{B_{\lambda 2 \rho} }g_u^q \, d\mu \right)^\frac{1}{q}+
%\left( \frac{c \lambda^q}{\textrm{cap}_q(N_{B_\rho}(u), B_{2\rho})} %\int_{B_{2\rho}} g_u^q \, d\mu\right)^{\frac{1}{q}}\\
&\qquad\leq c \left( \frac{1}{\textrm{cap}_q(N_{B_\rho}(f), B_{2\rho})} \int_{B_{2\lambda\rho}} g_f^q \, d\mu\right)^{\frac{1}{q}},
\end{align*}
for any $0<\rho<(1/8)$diameter$(X)$, where $c=c(c_\mu,c_P,\lambda)$. Assume then that $f_{B_{2\rho}}=0$. Then we may directly use the $(2,q)$-Poincaré inequality together with Lemma \ref{estimate of capacity} to obtain the result.
\end{proof}
 \end{theorem} 
 
 When considering the variational capacity of the zero set of $u-\eta$ in a metric ball overlapping the lateral boundary, the regularity of the lateral boundary of $\Omega_T$ in the sense of variational capacity comes into play. In order to build this into the assumption of the set $\Omega$, we introduce the following.
 \begin{definition}\label{thickness}
 A set $E\subset X$ is said to be uniformly $p$-thick, if there exist positive constants $\delta$ and $\rho_0$ so that
 \begin{align*}
 \textrm{cap}_p(E\cap B_\rho(x), B_{2\rho}(x))\geq \delta \textrm{cap}_p( B_\rho(x), B_{2\rho}(x)),
 \end{align*}
 for every $x\in E$ and $0<\rho<\rho_0$.  %We refer to $\delta$ as the thickness constant and to $\rho_0$ as the radius of thickness.
 \end{definition}
The uniform $p$-thickness satisfies the following deep self improving property, which will be needed when showing the reverse Hölder inequality.%enable us to obtain a $(2,q)$-version of the Poincaré inequality with capacity, where $1<q<2$. 
 \begin{theorem}\label{self improving thickness} Let $X$ be a proper linearly locally convex metric space endowed with a doubling regular Borel measure, supporting a $(1,q_0)$-Poincaré inequality for some $1\leq q_0<\infty$. Let $p>q_0$ and suppose $E\subset X$ is uniformly $p$-thick. Then there exists $q<p$ so that $E$ is uniformly $q$-thick.
 \begin{proof}
 See \cite{BjorMacmShan01}.
 \end{proof}
 \end{theorem}
It is known, see Lemma 4.4 in \cite{ATG}, that a complete metric measure space equipped with a doubling measure is proper, and our assumptions for $X$ are sufficient for using Theorem \ref{self improving thickness}.
 
 Now we can begin to build the estimates needed for the reverse Hölder inequality. As before, we start by choosing a convenient test function in the definition of parabolic quasiminimizers and derive the fundamental energy estimate. Notice that the following is a  quite general estimate, we do not yet need the condition $B_{2\rho}(x_0)\setminus \Omega \neq 0$.
% \subsection{Standing assumptions}
 %Throughout the rest of this section we assume that $u\in L_\trm{loc}^2(0,T;N_\trm{loc}^{1,2}(\Omega))$ is a parabolic quasiminimizer in $\Omega_T$, and that $u$ satisfies a boundary and initial condition with a function $\eta\in W^{1,2}(0,T; N^{1.2}(\Omega))$ in the sense that

 \begin{lemma}[Energy estimate]\label{energy estimate near the lateral boundary}
% Consider the function $\phi=\varphi(u-\eta)\chi_{(0,t')}$, where $0\leq\varphi\leq 1$ is any function in $C^\infty(0,T; N_0^{1,2}(B))$, and $B\subset X$. 
There exists a positive constant $c=c(K)$, such that
 \begin{align*}
& \esssup_{t\in \Lambda_\rho \cap (0,T)}\int_{B_\rho\cap\Omega}|u(x,t)-\eta(x,t)|^2 \,d\mu + \int_{Q_\rho\cap \Omega_T} g^2_u \, d\nu\\
 &\quad\leq c \int_{(Q_\sigma\setminus Q_\rho)\cap \Omega_T} g_{u}^2\, d\nu +c\int_{Q_\sigma\cap \Omega_T} |u-\eta|^2 \left(1+\frac{1}{(\sigma-\rho)^2} \right) \, d\nu\\
& \quad+c\int_{Q_\sigma\cap \Omega_T} \left(g_\eta^2+\left| \frac{\partial \eta}{\partial t} \right|^2\right)\, d\nu.
 \end{align*}
 
 \begin{proof}
Assume $Q_\rho=B_\rho(x_0) \times \Lambda_\rho(t_0)$ such that $\Lambda_\rho(t_0) \cap (0,T)\neq \emptyset$ and $\rho<\sigma$. Let $t'\in \Lambda_\rho \cap (0,T)$, and define
 \begin{align*}
 %&\chi_h(t)=
 %\begin{cases}
 %1, & h\leq t \leq t'-h,\\
  %0, &  t \leq h/2 \textrm{ or }t\geq t'-h/2,
 %\end{cases}\\
%&|\chi_h'(t)| \leq \frac{2}{h} \textrm{ for every }t\in(0,T),\\
 &\chi_h(t) = 
 \begin{cases}
 \frac{t-h}{h}, &h\leq t \leq 2h, \\
 1,& 2h \leq t \leq t'-2h,\\
 \frac{t'-h-t}{h},&t'-2h\leq t \leq t'-h,\\
 0,&\textrm{otherwise}.
\end{cases}
% \textrm{ in }W^{1,2}(0,T) \textrm{ as } h\rightarrow 0.
 \end{align*}
 Let $\varphi\in C^\infty(0,T;N_0^{1,2}(B_\sigma))$, $0\leq \varphi\leq 1$, be such that  $\varphi=1$ in $B_\rho$, and 
\begin{align}\label{estimate for varphi}
g_\varphi^2+\left| \frac{\partial \varphi}{\partial t} \right| \leq \frac{2}{(\sigma-\rho)^2}. 
\end{align}
Consider the function $\phi=-\varphi (u_\varepsilon-\eta_\varepsilon) \chi_{h,\delta}$. Again, since $u$ is a parabolic quasiminimizer in $\Omega_T$ and $\phi$ has the required smoothness for a test function, we can insert $\phi$ in inequality \eqref{quasiminimizer inequality} and examine the resulting terms.
% Let $h>0$. Since $u$ is a parabolic quasiminimizer, the function $u(\cdot,\cdot-s)$ is also a parabolic quasiminimizer for any $s<h/2$. Multiplying the quasiminimizer inequality for $u(\cdot,\cdot-s)$ with a  standard time mollifier in the variable $s$, integrating the resulting inequality over the real axis with respect to the variable $s$, and then using Fubini's theorem yields now for every small enough $h>0$,  and $\varepsilon>0$, 
% \begin{equation}\label{quasiminizing inequality}
% \begin{split}
% -\int_{\{\phi_{\varepsilon,h} \neq 0\}} u_\varepsilon \phi_{\varepsilon,h} '\,d\nu+&\left( \int_{\{\phi_{\varepsilon,h}  \neq 0\}} g_{u(\cdot,\cdot-s)}^2\, d\nu\right)_\varepsilon\\
% &\leq cQ \left( \int_{\{\phi_{\varepsilon,h}  \neq 0\}} g_{u(\cdot,\cdot-s)-\phi_{\varepsilon,h} }^2\,d\nu \right)_\varepsilon .
% \end{split}
% \end{equation}
We begin by examining the first term on the left hand side. After adding and substracting $\eta_\varepsilon$ and then  conducting partial integration with respect to time, we can write
\begin{align*}
-&\int_{\{\phi\neq 0\}} \parts{u_\varepsilon}{t} \phi\,d\nu\\
%&= \int_{\{\phi_{\varepsilon,h} \neq 0\}} (u_\varepsilon-\eta_\varepsilon)'\varphi(u_\varepsilon-\eta_\varepsilon) \chi_h \,d\nu
%-\int_{\{\phi_{\varepsilon,h} \neq 0\}} \eta_\varepsilon' \varphi(u_\varepsilon-\eta_\varepsilon) \chi_h \,d\nu\\
&=\int_{\{\phi\neq 0\}}\frac{1}{2} \frac{\partial}{\partial t}((u_\varepsilon-\eta_\varepsilon)^2) \varphi\chi_{h} \,d\nu+\int_{\{\phi\neq 0\}}\frac{\partial \eta_\varepsilon}{\partial t} \varphi(u_\varepsilon-\eta_\varepsilon) \chi_{h} \,d\nu.
\end{align*}
Performing partial integration on the first term on the right hand side yields now
\begin{align*}
-&\int_{\{\phi\neq 0\}} \parts{u_\varepsilon}{t} \phi\,d\nu\\
&=-\frac{1}{2h}\int_{h}^{2h}\int_{B_\sigma\cap \Omega}  (u_\varepsilon-\eta_\varepsilon)^2 \varphi\,d\mu\,dt+\frac{1}{2h}\int_{t'-2h}^{t'-h} \int_{B_\sigma\cap \Omega} (u_\varepsilon-\eta_\varepsilon)^2 \varphi\,d\mu\,dt\\
&\quad-\int_{\{\phi\neq 0\}}\frac{1}{2} (u_\varepsilon-\eta_\varepsilon)^2 \frac{\partial \varphi}{\partial t}\chi_{h} \,d\nu-\int_{\{\phi\neq 0\}} \frac{\partial \eta_\varepsilon}{\partial t} \varphi(u_\varepsilon-\eta_\varepsilon) \chi_{h} \,d\nu.
\end{align*}
Hence, after taking the limit $\varepsilon\rightarrow 0$ and then the limit $h \rightarrow 0$ we have the following: For almost every $0<t'<T$, using the initial condition \eqref{initial condition 2} yields
\begin{align*}
\lim_{\varepsilon,h\rightarrow 0}-\int_{{\{\phi\neq 0\}}} u_\varepsilon \frac{\partial \phi}{\partial t}\,d\nu
%&= \int_{\{\phi_{\varepsilon} \neq 0\}} \varphi\frac{1}{2}\left( |u_\varepsilon-\eta_\varepsilon|^2\right)' \,d\nu +\int_{\{\phi_{\varepsilon} \neq 0\}} \eta_\varepsilon ' \varphi(u_\varepsilon-\eta_\varepsilon)\,d\nu\\
\geq\frac{1}{2}\int_{B_\rho\cap \Omega}(u(x,t')-\eta(x,t'))^2 \varphi(x,t')\,d\mu\\
-\int_{Q_\sigma \cap\Omega_T}\frac{1}{2} (u-\eta)^2 \left|\frac{\partial \varphi}{\partial t}\right| \,d\nu-\int_{Q_\sigma\cap\Omega_T} \left|\frac{\partial \eta}{\partial t}\right| \varphi \left|u-\eta\right|  \,d\nu.
\end{align*}
Also,
\begin{align*}
&\lim_{\varepsilon,h \rightarrow 0}\int_{\{\phi\neq 0\}} (g_u^2)_{\eps}\,d\nu\geq \int_{Q_\rho\cap \Omega_T} g_{u}^2\,d\nu.
\end{align*}
On the right hand side of \eqref{quasiminimizer inequality}, $g_{(u-
\eta)-(u-\eta)_\eps}^2\rightarrow 0$  and $(g_{u(\cdot,\cdot-s)-u}^2)_\eps\rightarrow 0$ in $L_\textrm{loc}^1(\Omega_T)$ as $\eps\rightarrow 0$. Hence
\begin{align*}
 &\limsup_{\varepsilon,h \rightarrow 0} \int_{\{\phi \neq 0\}}\left( g_{u(\cdot,\cdot-s)-\phi}^2\right)_\varepsilon\,d\nu \leq c \int_{Q_\sigma \cap \Omega_T} g_{u-\varphi(u-\eta)}^2\,d\nu\\
& \leq c \int_{Q_\sigma \cap \Omega_T}( g_{(1-\varphi)(u-\eta)}^2+g_{\eta}^2)\,d\nu \leq c \int_{Q_\sigma \cap \Omega_T} (1-\varphi)^2\left(g_{u}^2+g_\eta^2\right)\,d\nu\\
&\quad+c \int_{Q_\sigma \cap \Omega_T}|u-\eta|^2g_\varphi^2\,d\nu+c \int_{Q_\sigma \cap \Omega_T} g_{\eta}^2\,d\nu.
\end{align*}
Noting that $\varphi=1$ in $Q_\rho$, combining all the obtained results together through \eqref{quasiminimizer inequality} and then using Young's inequality and \eqref{estimate for varphi} yields us the desired expression.  
\end{proof}
\end{lemma}
Having established the fundamental energy estimate, we derive from it a Caccioppoli inequality by using the hole filling iteration. As in section \ref{away from the lateral boundary}, we use the fundamental energy estimate for $\rho<\sigma<2\rho$.
 \begin{theorem}[Caccioppoli] \label{Caccioppoli near the lateral boundary} There exists a positive constant $c=c(K)$, such that
 \begin{align*}
 \int_{Q_\rho \cap \Omega} g_u^2\, d\nu \leq c\int_{Q_{2\rho}\cap \Omega} |u-\eta|^2\left(1+\frac{1}{\rho^2}\right)\,d\nu%\\
 +c\int_{Q_{2\rho}\cap \Omega} \left(g_\eta^2+\left| \frac{\partial \eta}{\partial t}\right|^2 \right) \,d\nu.
 \end{align*}
 \begin{proof} %Assume a parabolic cylinder $Q_\rho$, and any $\rho<\sigma<2\rho$. Let $0\leq \varphi\leq 1$, $\varphi \in C^\infty(0,T; N_0^{1,2}(\Omega))$ be such that
% \begin{align*}
%& \varphi=1\textrm{ in }Q_\rho,\\
% &0\leq \varphi<1\textrm{ in }Q_\sigma\setminus Q_\rho,\\ 
% &\textrm{supp}(\varphi) \subset Q_{\sigma},\\
% &\left| \frac{\partial \varphi}{\partial t}\right|+g_\varphi^2 \leq \frac{2}{(\sigma-\rho)^2}.
% \end{align*}
 After adding  $c\int_{Q_\rho \cap \Omega_T} g^2_u \, d\nu$ to both sides of the expression in Lemma \ref{energy estimate near the lateral boundary}, and then dividing by $c+1$, we can write
 \begin{align*}
 & \int_{Q_\rho \cap \Omega_T} g^2_u \, d\nu\\
 &\quad\leq \frac{c}{c+1} \int_{(Q_\sigma\setminus Q_\rho)\cap \Omega_T}
 g_{u}^2\, d\nu +\frac{c}{c+1}\int_{Q_\sigma \cap \Omega_T } |u-\eta|^2 \left(1+\frac{2}{(\sigma-\rho)^2}\right) \, d\nu\\
 &\qquad+\frac{c}{c+1}\int_{Q_\sigma \cap \Omega_T } \left(g_\eta^2+\left| \frac{\partial \eta}{\partial t} \right|^2\right)\, d\nu.
 \end{align*}
Then we choose 
 \begin{align*}
 \rho_0=\rho,\quad \rho_i-\rho_{i-1}=\frac{1-s}{s}s^i \rho, \quad i=1,2,\dots,k\quad s^2=\frac{1}{2}\left(\frac{c}{c+1}+1\right),
 \end{align*}
 replace $\rho$ by $\rho_{i-1}$ and $\sigma$ by $\rho_i$, and iterate to obtain
 \begin{align*}
 &\int_{Q_\rho \cap \Omega_T}g_u^2\,d\mu\,dt \leq  \left(\frac{c}{c+1}\right)^k \int_{Q_{\rho_k}\cap \Omega_T} g_u^2 \,d\nu\\
 &\qquad+\sum_{i=1}^{k}\left(\frac{c}{c+1}\right)^i\left(\int_{Q_{\rho_i} \cap \Omega_T } |u-\eta|^2 \left(1+\frac{2}{(\rho_i-\rho_{i-1})^2}\right) \, d\nu\right.\\
&\qquad\left. +\int_{Q_{\rho_i} \cap \Omega_T } \left(g_\eta^2+\left| \frac{\partial \eta}{\partial t} \right|^2\right)\, d\nu\right).
 \end{align*}
 Now, taking the limit $k\rightarrow \infty$ leads to the expression
 \begin{align*}
 \int_{Q_\rho \cap \Omega} g_u^2\, d\nu \leq c\int_{Q_{2\rho}\cap \Omega} |u-\eta|^2\left(1+\frac{1}{\rho^2}\right)\,d\nu%\\
 +c\int_{Q_{2\rho}\cap \Omega} \left(g_\eta^2+\left| \frac{\partial \eta}{\partial t}\right|^2 \right) \,d\nu,
 \end{align*}
 where $c=c(K)$.
 \end{proof}
\end{theorem}

 Then we prove a parabolic version of the Poincaré inequality for the function $u-\eta$, in the vicinity of the lateral boundary. We use the energy estimate with $\sigma=2\rho$. This is the stage at which we need the assumption that $x_0$ and $\rho$ are such that $B_{2\rho}(x_0)\setminus\Omega\neq \emptyset$, and moreover we need to be close enough to the lateral boundary, i.e. that $\rho$ is small enough. These assumptions enable us to exploit the uniform thickness of $\Omega$ to obtain an upper gradient estimate for the integral average of $u-\eta$.

 \begin{theorem}[Parabolic Poincaré]\label{parabolic Poincaré}
Assume $X\setminus \Omega$ is uniformly $2$-thick. Then there exist a positive constant $\rho_0<(1/8)$diam$(X)$ and a positive constant $c=c( c_\mu,c_P,\lambda, K)$, such that for every $0<\rho<\rho_0$ and parabolic cylinder $Q_\rho=B_\rho(x_0)\times \Lambda_{\rho}(t_0)$ such that $B_{2\rho}(x_0)\setminus \Omega \neq \emptyset$, we have
 \begin{align*}
 \esssup_{t\in \Lambda_\rho \cap (0,T)} \int_{B_\rho\cap \Omega}|u-\eta |^2 \, d\mu \leq c \int_ {Q_{6 \lambda\rho}\cap \Omega_T}\left(g_u^2+g_\eta^2 +\left|\frac{\partial \eta}{\partial t}\right|^2\right)\, d \nu.
 \end{align*}
\begin{proof} %Assume a parabolic cylinder $Q_\rho=B_\rho \times \Lambda_\rho$. Let $0\leq \varphi\leq 1$, $\varphi \in C^\infty(0,T; N_0^{1,2}(\Omega))$ be such that
% \begin{align*}
%& \varphi=1\textrm{ in }Q_\rho,\\
% &0\leq \varphi<1\textrm{ in }Q_{2\rho}\setminus Q_\rho,\\ 
% &\textrm{supp}(\varphi) \subset Q_{2\rho},\\
% &\left| \frac{\partial \varphi}{\partial t}\right|+g_\varphi^2 \leq \frac{2}{\rho^2}.
% \end{align*}
Assume a cylinder $Q_\rho=B_\rho(x_0)\times \Lambda_\rho(t_0)$, where $x_0\in X$ is such that $B_{2\lambda\rho}(x_0)\setminus \Omega \neq \emptyset$. From Lemma \ref{energy estimate near the lateral boundary}, we have 
 \begin{align*}
 &\esssup_{\Lambda_\rho \cap (0,T)} \int_{B_\rho \cap \Omega} | u-\eta |^2\, d\mu\leq c \int_{Q_{2\rho}\cap \Omega_T} g_u^2 \, d \nu \\&\quad+   c\int_{Q_{2\rho}\cap \Omega_T} |u-\eta|^2\left(1+\frac{2}{\rho^2}\right)\, d\nu+ c \int_{Q_{2\rho}\cap \Omega_T} \left(g_\eta^2+\left|\frac{\partial \eta}{\partial t}\right|^2\right) \, d \nu,
 \end{align*}
where $c=c(K)$. For $0<\rho < M$, we can estimate $1\leq  M^2/\rho^2$ on the right hand side to obtain
 \begin{equation}\label{parabolic poincare after estimation}
 \begin{split}
&\esssup_{\Lambda_\rho \cap (0,T)} \int_{B_\rho \cap \Omega} | u-\eta |^2\, d\mu\leq c \int_{Q_{2\rho}\cap \Omega_T} g_u^2 \, d \nu\\&\quad+ \frac{c}{\rho^2} \int_{Q_{2\rho}\cap \Omega_T} |u-\eta|^2\, d\nu+c \int_{Q_{2\rho}\cap \Omega_T} \left(g_\eta^2+\left|\frac{\partial \eta}{\partial t}\right|^2\right) \, d \nu,
 \end{split}
 \end{equation}
where $c=c(K,M)$. Next we continue the mapping $u(\cdot, t)-\eta (\cdot, t)$ outside of $\Omega$ by setting $u(\cdot, t)-\eta(\cdot, t) =0$ in $X \setminus \Omega$. %$2$-quasi everywhere. 
By assumption $B_{2 \rho}(x_0)\setminus \Omega \neq \emptyset$, and so there exists a point $x'\in X\setminus \Omega$, such that $B_{2 \rho}(x_0)\subset B_{4 \rho}(x')$ and $B_{4\lambda \rho}(x')\subset B_{6\lambda \rho}(x_0)$. 
 Since $X\setminus \Omega$ is uniformly $2$-thick, and then by Lemma \ref{estimate of capacity}, there exist positive constants $\rho_0<(1/8)$diameter$(X)$ and  $c=c(c_P, c_\mu)$, such that
 \begin{align*}
 \textrm{cap}_2(N_{B_{2\rho}(x')}(u-\eta),& B_{4\rho}(x'))
 \geq \textrm{cap}_2 ((X\setminus \Omega) \cap B_{2\rho}(x'), B_{4\rho}(x'))\\
 &\geq \delta \textrm{cap}_2( B_{2\rho}(x'), B_{4\rho}(x'))\geq\delta  c\frac{\mu (B_{2\rho}(x'))}{\rho^2},
 \end{align*}
 for every $0<\rho<\rho_0$. Hence, after using  Lemma \ref{Poincaré with capacity}, we can estimate that for any $0<\rho<\rho_0$
\begin{align*}
 &\frac{1}{\rho^2} \int_{Q_{2\rho}\cap \Omega_T} |u-\eta|^2\, d\nu\leq \frac{1}{\rho^2}  \int_{\Lambda_{2\rho}\cap (0,T)} \mu(B_{4\rho}(x')) 
\kint_{B_{4 \rho}(x')}|u-\eta |^2 \, d\mu\,dt\\
&\leq \frac{c}{\rho^2}  \int_{\Lambda_{2\rho}\cap (0,T)} \frac{\mu(B_{4\rho}(x')) }{\textrm{cap}_2 (N_{B_{2\rho}(x')} (u-\eta), B_{4\rho}(x'))} \int_{B_{4\lambda\rho}(x')} g_{u-\eta}^2 \, d\mu\,dt\\
%&\quad\leq c \int_{\Lambda_{2\rho}\cap (0,T)} \int_{B_{10\lambda^2\rho}(x_0)\cap \Omega} g_{u-\eta}^2\, d\mu\,dt\\
&\qquad
\leq c \int_{Q_{6\lambda\rho}\cap \Omega_T}  g_{u}^2\, d\nu+c \int_{Q_{6\lambda\rho}\cap \Omega_T}  g_{\eta}^2\, d\nu,
 \end{align*}
where $c=c(c_\mu,c_P,\delta)$. Here we also used the fact that $g_{u-\eta}(\cdot,t)=0$ $\mu$-almost everywhere outside of $\Omega$. Plugging this into \eqref{parabolic poincare after estimation}  completes the proof.
 \end{proof}
 \end{theorem}

Now we start from the Caccioppoli inequality, and then combine the parabolic Poincaré inequality together with the (2,q)-Poincaré with capacity to obtain a reverse Hölder inequality. We use the self improving property of the uniform thickness to have control over the variational capacity of the zero set of $u-\eta$. Again, we need the assumption that we are close enough to the lateral boundary, in other words that $\rho$ is small enough and $B_\rho(x_0)\setminus \Omega \neq \emptyset$.

 \begin{theorem}[Reverse Hölder inequality] \label{reverse hölder lateral} Suppose that $X\setminus \Omega$ is uniformly $2$-thick. Then there exist positive constants  $\rho_0<(1/8)$diameter$(X)$, $c=c(c_\mu,c_P,\lambda, K)$  and $1<q_0<2$ such that for every $0<\rho<\rho_0$ and $Q_\rho=B_\rho(x_0) \times \Lambda_\rho(t_0)$ such that $B_{2\rho}(x_0)\setminus  \Omega\neq \emptyset$, we have
\begin{align*}
 &\frac{1}{\nu(Q_{\rho})}\int_{Q_{\rho}\cap \Omega_T}g_u^2 \, d\nu\\
 &\quad\leq \varepsilon\frac{c}{\nu(Q_{6 \lambda\rho})}  \int_{Q_{6\lambda\rho}\cap \Omega_T} g_u^2 \, d\nu+\varepsilon^{-1}c\left(\frac{1}{\nu(Q_{6\lambda\rho})}\int_{Q_{6\lambda\rho}\cap \Omega_T} g_{u}^q \, d\nu\right)^\frac{2}{q}\\
 &\qquad+(\varepsilon^{-1}+\varepsilon)\frac{c}{\nu(Q_{6\lambda\rho})}  \int_{Q_{6\lambda\rho}\cap \Omega_T} \left(g_\eta^2+\left| \frac{\partial \eta}{\partial t} \right|^2 \right)\, d\nu,
 \end{align*}
 for any positive $\eps$ and $q_0\leq q$.
\begin{proof}Let $Q_\rho=B_\rho(x_0)\times \Lambda_\rho(t_0)$, such that  $B_{2\lambda\rho}(x_0)\setminus  \Omega\neq \emptyset$. 
From Theorem \ref{Caccioppoli near the lateral boundary} we know that for every $0<\rho<M$, we have
 \begin{align*}
 &\frac{1}{\nu(Q_\rho)} \int_{Q_\rho \cap \Omega_T}g_u^{2} \, d\nu\\
 &\quad\leq \frac{c}{\rho^2\nu(Q_{2\rho})}\int_{Q_{2\rho} \cap \Omega_T} |u-\eta|^2 \, d\nu+\frac{c}{\nu(Q_{2\rho})} \int_{Q_\rho \cap \Omega_T}\left( \left|\frac{\partial \eta}{\partial t}\right|^2+g_\eta^2\right)\, d\nu,
 \end{align*}
where $c=c(K,M)$. Similarly to what was done earlier, for almost every $t\in (0,T)$, we continue the mapping $u(\cdot, t)-\eta(\cdot,t)$ to be zero outside $\Omega$, and so 
from now on the mapping $u(\cdot, t)-\eta(\cdot,t)$ is to be thought of as defined in the whole space $X$. Also, since by assumption $B_{2 \rho}(x_0)\setminus \Omega \neq \emptyset$, there exists a point $x'\in X\setminus \Omega$, such that $B_{2\rho}(x_0)\subset B_{4 \rho}(x')$, and $B_{4\lambda \rho}(x')\subset B_{6\lambda \rho}(x_0)$. Let $1<q<2$ be as in Lemma \ref{Poincaré with capacity}. We have 
\begin{align*}
&\quad\frac{1}{\rho^2 \nu(Q_{2\rho})}\int_{Q_{2\rho}\cap \Omega_T} |u-\eta|^2\,d\nu\\
&\leq\frac{1}{\rho^4} \int_{\Lambda_{2\rho}\cap (0,T)} 
\left(\vint_{B_{2\rho}}|u-\eta|^2\,d\mu\right)^{1-\frac{q}{2}}\left(\vint_{B_{2\rho}}|u-\eta|^2 \, d\mu \right)^{\frac{q}{2}}\,dt\\
&\leq \frac{c}{\rho^4}\left( \esssup_{\Lambda_{2\rho}\cap (0,T)}\vint_{B_{2\rho}}|u-\eta|^2\,d\mu\right)^{1-\frac{q}{2}}\int_{\Lambda_{2\rho}\cap (0,T)}\left(\vint_{B_{4\rho}(x')}|u-\eta|^2 \, d\mu \right)^{\frac{q}{2}}\,dt,
\end{align*}
where $c=c(c_\mu,\lambda)$. In the above expression the former factor on the right hand side can be estimated by Theorem \ref{parabolic Poincaré}, and the latter factor by Theorem \ref{Poincaré with capacity}. We obtain that for some $1< q_0< 2$, we have for every $q_0 \leq q \leq 2$ and $0<2 \rho<(1/8)$diameter$(X)$,
\begin{align*}
&\frac{1}{\rho^2\nu(Q_{2\rho})}\int_{Q_{2\rho}\cap \Omega_T}|u-\eta|^2 \, d\nu\\
&\leq \frac{c}{\rho^2}\left\{ \frac{\rho^2}{\nu(Q_{6 \lambda\rho})}\int_{Q_{6\lambda\rho}\cap \Omega_T}\left( g_{u}^2+ g_\eta^2 +\left| \frac{\partial \eta}{\partial t} \right|^2 \right)\, d\nu \right\}^{1-\frac{q}{2}}\\
&\qquad\cdot \vint_{\Lambda_{2\rho}\cap (0,T)} \frac{1}{\textrm{cap}_q(N_{B_{2\rho}(x')}(u-\eta),B_{4\rho}(x'))} \int_{B_{4\lambda\rho}(x')} g_{u-\eta}^q \, d\mu\,dt,
\end{align*}
where $c=c(c_\mu, c_P,K)$. By assumption, $X\setminus \Omega$ is uniformly $2$-thick. By Theorem \ref{self improving thickness} this implies that for some $1<q<2$ the set $X\setminus\Omega$ is also uniformly $q$-thick with some positive constant $\delta$. We use this with Lemma \ref{estimate of capacity}, to conclude that 
\begin{align*}
\textrm{cap}_q(N_{B_{2\rho}(x')}(u-\eta), B_{4\rho}(x'))&\geq \textrm{cap}_q(X\setminus\Omega \cap B_{2\rho}(x'), B_{4\rho}(x')) \\
&\geq\delta \textrm{cap}_q(B_{2\rho}(x'),B_{4\rho}(x'))\\
&\geq \delta c\frac{\mu(B_{2\rho}(x'))}{\rho^q},
\end{align*}
  for every $0<\rho<\rho_0<(1/8)$diameter$(X)$, where $c=c(c_\mu,c_P)$. Hence for every $\rho<\rho_0$, we obtain
 \begin{align*}
 &\frac{c}{\rho^2\nu(Q_{2\rho})}\int_{Q_{2\rho}\cap \Omega_T}|u-\eta|^2 \, d\nu\\
 &\leq \left\{ \frac{c}{\nu(Q_{6\lambda\rho})}\int_{Q_{6\lambda\rho}\cap \Omega_T} \left(g_{u}^2+ g_\eta^2 +\left| \frac{\partial \eta}{\partial t} \right|^2\right) \, d\nu \right\}^{1-\frac{q}{2}}\\
 &\cdot \frac{c}{\nu(Q_{6\lambda\rho})}\int_{Q_{6\lambda\rho}\cap \Omega_T} g_{u-\eta}^q \, d\nu.
 \end{align*}
 Now we can use the $\varepsilon$-Young inequality and then Hölder's inequality to conclude that for every positive $\varepsilon$ we have the estimate
 \begin{align*}
 &\frac{1}{\nu(Q_{\rho})}\int_{Q_{\rho}\cap \Omega_T}g_u^2 \, d\nu\\
 &\quad\leq \varepsilon\frac{c}{\nu(Q_{6 \lambda\rho})}  \int_{Q_{6\lambda\rho}\cap \Omega_T} g_u^2 \, d\nu+\varepsilon^{-1}c\left(\frac{1}{\nu(Q_{6\lambda\rho})}\int_{Q_{6\lambda\rho}\cap \Omega_T} g_{u}^q \, d\nu\right)^\frac{2}{q}\\
 &\qquad+(\varepsilon^{-1}+\varepsilon)\frac{c}{\nu(Q_{6\lambda\rho})}  \int_{Q_{6\lambda\rho}\cap \Omega_T} \left(g_\eta^2+\left| \frac{\partial \eta}{\partial t} \right|^2 \right)\, d\nu.
 \end{align*}
 where the positive constant $c=c(c_\mu, c_P, \lambda, K)$. %Using now Lemma \ref{global  gehring} below with $A=4$, and setting
%\begin{align*}
%f_1=g_\eta^2+\left| \frac{\partial \eta}{\partial t} \right|^2,  \qquad f_2=0,
%\end{align*}  
%  we obtain, after estimating on the right hand side, the desired result.
\end{proof}
 \end{theorem}

\section{Global higher integrability}
In this section we cover the steps from the reverse Hölder inequality both near and away from the lateral boundary of $\Omega_T$, to the global higher integrability of a parabolic quasiminimizer's upper gradient.

We begin by proving a modification of Gehring's Lemma in metric spaces, to take into account the terms in the reverse Hölder inequalities that result from the initial and lateral boundary conditions. 

In the proof of this theorem the initial cylinder is divided into a good set where $g$ is bounded and into a bad set where $g$ is unbounded. At each point of the bad set, in some small enough cylinder centered at this point, we have by our previous results a reverse Hölder inequality. These cylinders are then used to form a Vitali covering of the bad set, so that we obtain the reverse Hölder inequality over to whole bad set. Finally, by an argument involving Fubini's theorem, the reverse Hölder inequality is used to establish higher integrability over the bad set.

\begin{theorem}
\label{global gehring}
Let $g\in L^2_{\trm{loc}}(0,T;\N_\trm{loc}(\Om))$ and let $f_1$, $f_2$ be non negative measurable functions defined in $\Omega_T$. Consider a parabolic cylinder $Q_{2R}(z_0)=B_{2R}(x_0)\times \Lambda_{2R}(t_0)\subset X\times \R$. Let $s$ be the constant from \eqref{eq:DoublingConsequence} and let $q$ be such that $2s/(2+s)<q<2$. Suppose that there exists a positive constant $A>1$, for which with any $z'=(x',t')$ and $\rho$ such that $Q_{A\rho}(z')\subset Q_{2R}(z_0)$, we have, after abbreviating $Q_\rho=Q_\rho(z')$, $Q_{A\rho}=Q_{A\rho}(z')$ and $B_{A\rho}=B_{A\rho}(x')$,
\begin{equation} 
\label{gekaehto}
\begin{split}
\frac{1}{\nu(Q_\rho)}&\int_{Q_\rho\cap \Omega_T}  g^2 \, d \nu
\leq
\eps\frac{1}{\nu(Q_{A\rho})}\int_{Q_{A \rho}\cap \Omega_T} g^2
\,d \nu\\&+\gamma\l(\frac{1}{\nu(Q_{A\rho})}\int_{Q_{A \rho}\cap \Omega_T} g^{{q}} \,d \nu \r)^{2/{q}}+\gamma\frac{1}{\nu(Q_{A\rho})}\int_{Q_{A \rho}\cap \Omega_T} f_1^2\,d \nu\\&\qquad+\gamma\left(\frac{1}{\mu(B_{A\rho})}\int_{B_{A \rho}\cap \Omega} f_2^q\,d \mu\right)^{2/q},
\end{split}
\end{equation}
for any $\eps>0$, where $\gamma$ may depend on $\eps$.
 Then there exists positive constants $\eps_0=\eps_0(c_\mu, A, \gamma, q)$  and $c=c(c_\mu, A,\gamma)$, such that 
\[
\begin{split}
&\left(\frac{1}{\nu(Q_R)}\int_{Q_{R} \cap \Omega_T} g^{2+\eps} \,d \nu\right)^\frac{1}{2+\eps}\leq  \l(\frac{c}{\nu(Q_{2R})}\int_{Q_{2R}\cap \Omega_T} g^2 \,d \nu\r)^{\frac{1}{2}}\\
&+\left(\frac{c}{\nu(Q_{2R})}\int_{Q_{2R}\cap \Omega_T} f_1^{2+\varepsilon}\,d \nu\right)^\frac{1}{2+\eps}+\left(\frac{c}{\mu(B_{2R})}\int_{B_{2R}\cap \Omega} f_2^{q+\varepsilon}\,d \mu\right)^{\frac{1}{q+\eps}},
\end{split}
\]
for every $0<\eps\leq \eps_0$, where we have abbreviated $Q_R=Q_R(z_0)$, $Q_{2R}=Q_{2R}(z_0)$ and $B_{2R}=B_{2R}(x_0)$.
\end{theorem}
%

%%%%%%%%%%%%%%%%%%%%%%%%%%%%%%%%%%%%%%%%%%%%%%%%%%%%%%%%%%%%%%%%%%%%%%%%%%%%%%%%%%%%%%%%%%%%%%%%%%%%%%%%%%%%%%%%%%%%%%%%%
\begin{proof}%Denote $Q_{2R}=Q_{2R}(z_0)$, where $R$ and $z_0$ are chosen so that $Q_{2R}\subset \Omega_T$.
Assume a parabolic cylinder $Q_{2R}$ with center point $z_0=(x_0,t_0)$. Define for every $z_1=(x_1,x_2)$, $z_2=(x_2,t_2)\in X\times \R$ the parabolic distance
\begin{align*}
 \textrm{dist}_p(z_1,z_2)=d(x_1,x_2)+|t_1-t_2|^{1/2}.
\end{align*}
Using this, set for every $z\in Q_{2R}$ the functions
\begin{align*}
r(z)&=\,\textrm{dist}_p(z,(X\times \R)\setminus Q_{2R}),\\
 \alpha(z)&=\frac{\nu(Q_{2R})}{\nu(Q_{\frac{r(z)}{5A}}(z))}.
\end{align*}
From the definition of $r(z)$ it can readily be checked that $Q_{r(z)}(z)\subset Q_{2R}$ for every $z\in Q_{2R}$.
%By the doubling property of the measure, there exists a positive constant $c_0=c_0(c_\mu, A)$, such that for every $z\in Q_{2R}$ 
%\begin{align*}
%\frac{\nu(Q_{2R})}{\nu(Q_{\frac{r(z)}{5A}}(z))}\leq c_0 \alpha(z). 
%\end{align*} 
%Let $Q_i=Q_{r_i}(z_i)$, $i=1,2,...,$ be the Whitney decomposition of $Q_{2R}$, with the properties as described in. For every $z\in Q_{2R}$, define
%\begin{align*}
%\alpha(z)&=\max\frac{\nu(Q_R)}{\{\nu(Q_i)\,:\,z\in Q_i\,\}},\\
%r(z)&=\max\{\,r_i.\,:\,z\in Q_i=Q_{r_i}(z_i)\,\}.
%\end{align*}
%For any $z\in Q_{2R}$ and any Whitney cylinder $Q_i=Q_{r_i}(z_i)$ which contains $z$, we have $Q_i\subset Q_{2r(z)}(z)$. Therefore, by the doubling property of the measure, there exists a positive constant $c_1=c_1(c_\mu, A)>0$, such that for every $z\in Q_{2R}$ and $r\in[r(z)/(5A),  r(z)]$,
%\begin{align*}
%\frac{\nu(Q_R)}{\nu(Q_r(z))}\leq \frac{\nu(Q_R)}{\nu(Q_{\frac{r(z)}{5A}}(z))}\leq c_1 \alpha(z).
%\end{align*}
For $z\in Q_{2R}$, define
\begin{align*}
h(z)=\alpha^{-1/2}(z)g(z),
\end{align*}  
and for every $\lambda>0$, set
\begin{align*}
G(\lambda)=\{\,z\in Q_{2R}\cap \Omega_T\,:\,h(z)>\lambda\,\}.
\end{align*}
Denote
\begin{align*}
\lambda_0=\left( \frac{1}{\nu(Q_{2R})}\int_{Q_{2R}\cap \Omega_T} g^2\,d\nu \right)^{1/2}.
\end{align*}
Assume $\lambda>\lambda_0$. For $\nu$-almost every $z'\in G(\lambda)$, we have for every \\$r\in [r(z')/(5A),r(z')]$, 
\begin{align}\label{larger than}
\frac{1}{\nu(Q_r(z'))}\int_{Q_r(z')\cap \Omega_T}g^2\,d\nu \leq \frac{\alpha(z')}{\nu(Q_{2R})} \int_{Q_{2R}\cap \Omega_T}g^2\,d\nu \leq \alpha(z') \lambda^2, 
\end{align}
and by the definition of $G(\lambda)$, since $\mu$ is a positive Borel measure,
\begin{align}\label{less than}
\lim_{r\rightarrow 0} \frac{1}{\nu(Q_r(z'))}\int_{Q_r(z')\cap \Omega_T} g^2\, d\nu=g^2(z')>\alpha(z') \lambda^2.
\end{align}
Now \eqref{larger than} and \eqref{less than} imply that for $\nu$-almost every $z'\in G(\lambda)$, there exists a corresponding radius $\rho(z')\in (0,r(z')/(5A))$, for which it holds 
\begin{equation}\label{comparability of balls for g}
\begin{split}
\frac{1}{\nu(Q_{A \rho(z')}(z'))}&\int_{Q_{A \rho(z')}(z')\cap \Omega_T} g^2\, d\nu\\
&\leq \alpha(z') \lambda^2\leq \frac{1}{\nu(Q_{ \rho(z')}(z'))}\int_{Q_{\rho(z')}(z')\cap \Omega_T} g^2\, d\nu. 
\end{split}
\end{equation}
%Therefore we are able to choose a positive $\delta=\delta(c_\mu, A)$, such that for %$\nu$-almost every $z'\in G(\lambda)$
%\begin{align*}
%\delta \vint_{Q_{A\rho(z')}(z')} g^2\, d\nu \leq \frac{1}{2} \vint_{Q_{\rho(z')}(z')}g^2\, d\nu.
%\end{align*}
Thus by choosing $\eps=1/2$ in \eqref{gekaehto}, we can absorb the first term on the right hand side of \eqref{gekaehto} into the left hand side and obtain
\begin{align*}
\frac{1}{\nu(Q_{\rho(z')}(z'))}&\int_{Q_{\rho(z')}(z')\cap \Omega_T} g^2 \, d\nu \leq  \left( \frac{c}{\nu(Q_{A\rho(z')}(z'))}\int_{Q_{A\rho(z')}(z')\cap \Omega_T}g^q\, d\nu \right)^{2/q}\\
&+\frac{c}{\nu(Q_{A\rho(z')}(z'))}\int_{Q_{A \rho(z')}(z')\cap \Omega_T} f_1^2\,d \nu\\
&\qquad+\left(\frac{c}{\mu(B_{A\rho(z')}(z'))}\int_{B_{A \rho(z')}(z')\cap \Omega} f_2^q\,d \mu\right)^{2/q},
\end{align*}
for $\nu$-almost every $z'\in G(\lambda)$,  where $c=c(\gamma)$. This together with \eqref{comparability of balls for g} yields
\begin{equation}\label{higherint of g when centered in G}
\begin{split}
\frac{1}{\nu(Q_{5A\rho(z')}(z'))}&\int_{Q_{5A\rho(z')}(z')\cap \Omega_T} g^2 \, d\nu \\&\leq \left( \frac{c}{\mu(Q_{A\rho(z')}(z'))}\int_{Q_{A\rho(z')}(z')\Omega_T}g^q\,d\nu \right)^{2/q}\\
&\quad+\frac{c}{\nu(Q_{A\rho(z')}(z'))}\int_{Q_{A \rho(z')}(z')\cap \Omega_T} f_1^2\,d \nu\\&\qquad+\left(\frac{c}{\mu(B_{A\rho(z')}(z'))}\int_{B_{A \rho(z')(z')}\cap \Omega} f_2^q\,d \mu\right)^{2/q},
\end{split}
\end{equation}
where $c=c(A,c_\mu,\gamma)$. From the definitions of a parabolic cylinder and the parabolic distance, it follows that
\begin{align*}
2^{-1/2} r(z')\leq r(z) \leq 2 r(z') \qquad \textrm{for every } z\in Q_{r(z')}(z'),\; z'\in Q_{2R}.
\end{align*}
From this it is straightforward to check that
\begin{align*}
\begin{array}{c}Q_{r(z)}(z)\subset Q_{3r(z')}(z'),\\ Q_{r(z')}(z')\subset Q_{4r(z)}(z)\end{array} \quad \textrm{ for every }  z\in Q_{r(z')}(z'),\; z'\in Q_{2R},
\end{align*} 
and so by the doubling property of the measure there exists positive constants $c=c(c_\mu)$, $c'=c'(c_\mu)$ such that
\begin{align}\label{local comparability of alpha}
c\alpha(z')\leq \alpha(z) \leq c' \alpha (z) \qquad \textrm{ for every }z\in Q_{r(z')}(z'),\; z'\in Q_{2R}.
\end{align}
Because of this, we see from \eqref{higherint of g when centered in G} that there exists a positive constant $c=c(A,c_\mu,\gamma)$, such that for $\nu$-almost every $z'\in G(\lambda)$, after also using the fact that $\alpha(z)\geq 1$,
\begin{align}\label{higherint of h when centered in G}
\begin{split}
\frac{1}{\nu(Q_{5A \rho(z')}(z'))}&\int_{Q_{5A \rho(z')}(z')\cap \Omega_T} h^2 \, d\nu \\ &\leq \left( \frac{c}{\nu(Q_{A\rho(z')}(z'))}\int_{Q_{A\rho(z')}(z')\cap \Omega_T}h^q\,d\nu \right)^{2/q}\\
&\quad+\frac{c}{\nu(Q_{A\rho(z')}(z'))}\int_{Q_{A \rho(z')}(z')\cap \Omega_T} f_1^2\,d \nu\\&\quad+\left(\frac{c}{\mu(B_{A\rho(z')}(z'))}\int_{B_{A \rho(z')}(z')\cap \Omega} f_2^q\,d \mu\right)^{2/q}.
\end{split}
\end{align}
On the other hand, by  Hölder's inequality since $1<q<2$, and then by \eqref{local comparability of alpha}, we obtain from \eqref{comparability of balls for g}, 
\begin{align}\label{q boundedness of h}
\begin{split}
&\left( \frac{1}{\nu(Q_{A\rho(z')}(z'))}\int_{Q_{A\rho(z')}(z')\cap \Omega_T}h^q\, d\nu \right)^{(2-q)/q}\\
&\quad\leq \left( \frac{1}{\nu(Q_{A\rho(z')}(z'))} \int_{Q_{A\rho(z')}(z')\cap \Omega_T}h^2\, d\nu \right)^{(2-q)/2}\leq c\lambda^{2-q},
\end{split}
\end{align}
where $c=c(c_\mu)$. Define 
\begin{align*}
G_{f_1}(\lambda)&=\{\,z\in Q_{2R}\cap \Omega_T\,:\, f_1 > \lambda \,\},\\
G_{f_2}(\lambda)&=\{\,z \in B_{2R}\cap \Omega\,:\, f_2 > \lambda \, \}.
\end{align*}
Assume now any $\delta>0$. By \eqref{higherint of h when centered in G} and by the definitions of $G(\delta\lambda)$, $G_{f_1}(\delta\lambda)$ and $G_{f_2}(\delta\lambda)$, we have for $\nu$-almost every $z'\in G(\lambda)$,
\begin{align*}
\frac{1}{\nu(Q_{5A\rho(z')}(z'))}&\int_{Q_{5A\rho(z')}(z') \cap \Omega_T} h^2\, d\nu \\&\leq c  \delta^2 \lambda^2+\left(\frac{c}{\nu(Q_{A\rho(z')}(z'))} \int_{Q_{A\rho(z')}(z')\cap G(\delta \lambda)}h^q\,d\nu \right)^{2/q}\\
&\quad+\frac{c}{\nu(Q_{A\rho(z')}(z'))}\int_{Q_{A \rho(z')}(z')\cap G_{f_1}(\delta \lambda)} f_1^2\,d \nu\\&\quad+\left(\frac{c}{\mu(B_{A\rho(z')}(z'))}\int_{B_{A \rho(z')}(z')\cap G_{f_2}(\delta \lambda)} f_2^q\,d \mu\right)^{2/q}.
\end{align*}
By \eqref{local comparability of alpha} and \eqref{comparability of balls for g}, we can now choose a small enough positive number $\delta(c_\mu, A, \gamma)<1$ to absorb the first term on the right hand side into the left hand side. We obtain a positive $c=c(A,c_\mu,\gamma)$, such that for $\nu$-almost every $z'\in G(\lambda)$ and any $\lambda>\lambda_0$, after using \eqref{q boundedness of h},
\begin{equation}\label{higherint2 of h}
\begin{split}
&\frac{1}{\nu(Q_{5A\rho(z')}(z'))}\int_{Q_{5A\rho(z')}(z')} h^2\, d\nu \\
%&\leq c\left(\frac{1}{\nu(Q_{A\rho(z')}(z'))} \int_{Q_{A\rho(z')}(z')\cap G(\delta \lambda)}h^q\,d\nu \right)^{2/q}\\
&\leq  \lambda^{2-q} \frac{c}{\nu(Q_{A \rho(z')}(z'))} \int_{Q_{A\rho(z')}(z')\cap G(\delta \lambda)}h^q\,d\nu\\
&\quad+\frac{c}{\nu(Q_{A\rho(z')}(z'))}\int_{Q_{A \rho(z')}(z')\cap G_{f_1}(\delta \lambda)} f_1^2\,d \nu\\&\quad+\left(\frac{c}{\mu(B_{A\rho(z')}(z'))}\int_{B_{A \rho(z')}(z')\cap G_{f_2}(\delta \lambda)} f_2^q\,d \mu\right)^{2/q}.
\end{split} 
\end{equation} 
%where in the last step we used \eqref{q boundedness of h}.

The collection $\{\,Q_{A\rho(z')}(z')\,:\,z'\in G(\lambda)\,\}$ is now an open cover of $G(\lambda)$. By the Vitali covering lemma, there exists a countable and pairwise disjoint subcollection $\{\,Q_{A\rho(z_i')}(z_i')\,:\,z_i'\in G(\lambda)\,\}_{i=1}^\infty$, such that 
\begin{align*}
G(\lambda)\subset \bigcup_{i=1}^\infty Q_{5 A \rho(z_i')}(z_i')\subset Q_{2R}.
\end{align*}
The last inclusion follows from the fact that $5A \rho(z)\leq r(z)$. This property is the reason why we introduced the number $5$ into the proof earlier. Now we can write for any $\lambda>\lambda_0$, after multiplying inequality \eqref{higherint2 of h} with $\nu(Q_{A\rho(z')}(z'))$ and using the doubling property of $\mu$,
\begin{equation*}
\begin{split}
&\int_{G(\lambda)}h^2\,d\nu\leq \sum_{i=1}^\infty \int_{Q_{5A\rho(z_i')}(z_i')} h^2\, d\nu\\
&\leq  \sum_{i=1}^\infty \left(c \lambda^{2-q}\int_{Q_{A\rho(z_i')}(z_i')\cap G(\delta \lambda)}h^q\,d\nu+c\int_{Q_{A \rho(z_i')}(z_i')\cap G_{f_1}(\delta \lambda)} f_1^2\,d \nu\right.\\
&\left.\qquad+c\frac{(\rho(z'_j))^2 \mu(B_{\rho(z_i')}(z_j'))}{\mu(B_{\rho(z_i')}(z_j'))^{2/q}}\left(\int_{B_{A \rho(z_i')}(z_i')\cap G_{f_2}(\delta \lambda)} f_2^q\,d \mu\right)^{2/q} \right),
\end{split}
\end{equation*}
where $c=c(c_\mu, A, \gamma)$. Since by assumption $s> 0$ in \eqref{eq:DoublingConsequence}, for each $i$ we have
\begin{align*}
(\rho(z_i'))^2\mu(B_{\rho(z_i')}(z_i'))^{1-2/q}&\leq c\left(\frac{\mu(B_{2R}(x_0))}{(2R)^s}\right)^{1-2/q} (\rho(z_i'))^{2+s(1-2/q)}\\
&\leq c (\mu(B_{2R}(x_0)))^{1-2/q} R^2,
\end{align*}
for every $2s/(2+s)< q<2$, where $c=c(c_\mu)$. Hence
\begin{equation}\label{higherint2 of h at lower levels} 
\begin{split}
\int_{G(\lambda)}h^2\,d\nu &\leq c\lambda^{2-q}\int_{G(\delta \lambda)} h^q\, d\nu + c\int_{G_{f_1}(\delta \lambda)} f_1^2 \,d \nu\\
&\quad+ c (\mu(B_{2R}(x_0)))^{1-2/q} R^2 \left( \int_{G_{f_2}(\delta \lambda)}f_2^q \, d\mu \right)^{2/q}.
\end{split}
\end{equation}
 From now on the higher integrability result is a consequence of \eqref{higherint2 of h at lower levels} and Fubini's theorem. To see this, we integrate
over $G(\lambda_0)$ and use Fubini's theorem to obtain
\begin{equation*}
%\label{isojako}
\begin{split}
\int_{G(\lambda_0)}&h^{2+\eps} \, d \nu=
\int_{G(\lambda_0)}\l(\int_{\lambda_0}^{h} \eps \lambda^{\eps-1} \,d
\lambda\,+(\lambda_0)^{\eps}\r)  h^2 \,d\nu
\\%&\hspace{-2 em}=
 %\eps \int_{\lambda_0}^{\infty}
% \lambda^{\eps-1}\int_{G(\lambda_0)}\chi_{\{\lambda< h\}}h^2  \,d \nu \,d \lambda\, +
%(\lambda_0)^{\eps} \int_{G(\lambda_0)} h^2 \,d \nu\\
&\hspace{-2 em}=
  \int_{\lambda_0}^{\infty}
\eps \lambda^{\eps-1}\int_{G(\lambda)}h^2  \,d \nu \,d \lambda\, +
(\lambda_0)^{\eps} \int_{G(\lambda_0)} h^2 \,d \nu,
%\hspace{-2 em}\leq c \int_{\lambda_0}^{\infty} \eps \lambda^{\eps-1+2-{q} } \int_{
%G(\delta \lambda)} h^{{q}} \,d \nu \,d \lambda+
%\lambda_0)^{\eps} \int_{G(\lambda_0)} h^2 \,d \nu,
\end{split}
\end{equation*}
and now by \eqref{higherint2 of h at lower levels}
\begin{align*}
 \int_{\lambda_0}^{\infty}
 \eps\lambda^{\eps-1}\int_{G(\lambda)}h^2  \,d \nu \,d \lambda\, \leq  c\int_{\lambda_0}^{\infty} \varepsilon \lambda^{\varepsilon+1-q} \int_{G(\delta \lambda)}h^q\,d\nu\, d\lambda\\
 +c\int_{\lambda_0}^{\infty}\eps\lambda^{\varepsilon-1} \int_{G_{f_1}(\delta \lambda)} f_1^2\, d\nu\,d\lambda\\
 +c (\mu(B_{2R}(x_0)))^{1-2/q} R^2 \int_{\lambda_0}^{\infty}\eps\lambda^{\varepsilon-1}\left( \int_{G_{f_2}(\delta \lambda)} f_2^q \, d\mu \right)^{2/q}\,d\lambda.
\end{align*}
By Fubini's
theorem again, we see that
\begin{align*}
 &\int_{\lambda_0}^{\infty} \eps
 \lambda^{\eps+1-{q} } \int_{G(\delta \lambda)} h^{{q}} \,d \nu\,d \lambda\, +
\lambda_0^{\eps} \int_{G(\lambda_0)} h^2 \,d \nu
\\&\quad= 
 \eps\int_{G(\delta \lambda_0)} \l(\int_{\lambda_0}^{h/\delta}
 \lambda^{\eps-1+2-{q} } \,d \lambda\r) h^{{q}} \,d \nu\, +
\lambda_0^{\eps} \int_{G(\lambda_0)} h^2 \,d \nu\\&\qquad\leq
\frac{\eps}{\delta^{2+\varepsilon-q}(\eps+2-{q}) } \int_{G( \lambda_0)}
 h^{\eps+2} \,d \nu+
\lambda_0^{\eps} \int_{G(\delta\lambda_0)} h^2 \,d \nu,
\end{align*}
where $c=c(A,c_\mu,\gamma)$.  Observe that in the last step we also used the fact that $h^{\eps+2}\leq
  \lambda_0^{\eps} h^2$  in $G(\delta
  \lambda_0)\setminus G(\lambda_0)$. In similar fashion we obtain
  \begin{align*}
  \int_{\lambda_0}^{\infty}\eps\lambda^{\varepsilon-1} \int_{G_{f_1}(\delta \lambda)} f_1^2\, d\nu\,d\lambda\leq \delta^{-\varepsilon} \int_{Q_{2R}\cap \Omega_T} f_1^{2+\varepsilon}\, d\nu,
  \end{align*}
  and
  \begin{align*}
  \int_{\lambda_0}^{\infty}\eps\lambda^{\varepsilon-1}&\left( \int_{G_{f_2}(\delta \lambda)} f_2^q \, d\mu \right)^{2/q}\,d\lambda\\ &\leq \left( \int_{G_{f_2}(\delta \lambda_ 0)} f_2^q \, d\mu \right)^{2/q-1}\delta^{-\varepsilon} \int_{B_{2R}\cap \Omega_T} f_2^{q+\varepsilon}\, d\mu\\
& \qquad \leq \delta^{-\eps}\mu(B_{2R}(x_0))^{\frac{\varepsilon}{q+\varepsilon}(2/q-1)}\left(\int_{B_{2R}\cap \Omega} f_2^{q+\varepsilon}\, d\mu\right)^{\frac{2+\eps}{q+\eps}},
  \end{align*}
where in the final step we have used Hölder's inequality. We can now choose a positive $\eps=\eps(c_\mu, A, \gamma, q)$ small enough
to absorb the term containing $h^{2+\eps}$ into the left hand side of \eqref{higherint2 of h at lower levels}, and
conclude that
\begin{equation}
\label{melkein}
\begin{split}
%\label{paaest}
&\int_{G(\lambda_0)} h^{2+\eps} \,d \nu  \leq c(\lambda_0)^{\eps}
\int_{G(\delta\lambda_0)} h^2 \,d \nu+c\int_{Q_{2R}\cap \Omega_T} f_1^{2+\varepsilon}\, d\nu\\
&\qquad+ c\mu(B_{2R}(x_0))^{(1-\frac{\eps}{q+\varepsilon})(1-2/q)}R^2\left(\int_{B_{2R}\cap \Omega} f_2^{q+\varepsilon}\, d\mu\right)^{\frac{2+\eps}{q+\eps}},
\end{split}
\end{equation}
where $c=(c_\mu, A,\gamma)$. In case the term containing $h^{2+\varepsilon}$ is infinite, we replace $h$ by $h_k=\min\{h,k\}$ where $k>\lambda$. Starting from \eqref{higherint2 of h at lower levels} we estimate that
\begin{equation}\label{higherint for cutoff}
\begin{split}
\int_{\{h_k> \lambda \}} h_k^{2-q} \, d\zeta & \leq c \lambda^{2-q} \int_{\{h_k> \delta \lambda \}} d \zeta+ c\int_{G_{f_1}(\delta \lambda)} f_1^2 \,d \nu\\
&\quad+ c (\mu(B_{2R}(x_0)))^{1-2/q} R^2 \left( \int_{G_{f_2}(\delta \lambda)}f_2^q \, d\mu \right)^{2/q}.
\end{split}
\end{equation} 
where $d\zeta= h^q \, d\nu$. Performing now as above the calculations involving Fubini's theorem yields
\begin{align*}
\int_{\{h_k> \lambda_0 \}} &h_k^{2+\varepsilon-q} \, d\zeta \leq \varepsilon c \int_{\{h_k> \lambda_0 \}} h_k^{2+\varepsilon-q} \, d\zeta +\lambda_0^\varepsilon \int_{\{h_k> \delta \lambda_0\}}h_k^{2-q}\, d\zeta\\
&+c\int_{Q_{2R}\cap \Omega_T} f_1^{2+\varepsilon}\, d\nu + c\mu(B_{2R}(x_0))^\frac{q-2}{q+\varepsilon}R^2\left(\int_{B_{2R}\cap \Omega} f_2^{q+\varepsilon}\, d\mu\right)^{\frac{2+\eps}{q+\eps}}.
\end{align*}
Now we can absorb the term containing $h_k^{2+\varepsilon-q}$ into the left hand side side, and finally let $k\rightarrow \infty$ to obtain \eqref{melkein}.

Finally, from the definitions of the parabolic distance and the parabolic cylinder, it is again straighforward to check that $Q_R\subset Q_{4 r(z)}(z)$ for every $z\in Q_R$. Hence, by the doubling property of the measure,
\begin{align*}\label{boundedness of alpha}
\alpha(z) \leq \frac{\nu(Q_{2R})}{\nu(Q_R)}\frac{\nu(Q_{4 r(z)}(z))}{\nu(Q_{\frac{r(z)}{5A}}(z))}\leq  c_1, \qquad \textrm{for every} \,z\in Q_R, 
\end{align*}
where $c_1=c_1(c_\mu,A)>0$. On the other hand, clearly $\alpha(z)\geq 1$ for every $z\in Q_{2R}$. Now \eqref{melkein} and the definition of $\lambda_0$ imply that
\begin{align*}
\int_{Q_R\cap \Omega_T} g^{2+\varepsilon} \, d\nu &\leq c_1^{\frac{2+\varepsilon}{2}}\left((\lambda_0)^\eps\int_{Q_R\setminus G(\lambda_0)} h^2 \, d\nu+\int_{G(\lambda_0)} h^{2+\eps} \, d\nu \right)\\
&\leq c\frac{1}{(\nu(Q_{2R}))^{\eps/2}}\left(\int_{Q_{2R}\cap \Omega_T} g^2 \, d\nu \right)^{\frac{2+\eps}{2}}
+c\int_{Q_{2R}\cap \Omega_T} f_1^{2+\varepsilon}\, d\nu\\
&\qquad+ c\mu(B_{2R}(x_0))^{\frac{q-2}{q+\eps}}R^2\left(\int_{B_{2R}\cap \Omega} f_2^{q+\varepsilon}\, d\mu\right)^{\frac{2+\eps}{q+\eps}},
\end{align*}
 where $c=c(c_\mu, A,\gamma)>0$. From this expression the proof can readily be completed.
\end{proof}

We have now all the necessary pieces to prove global higher integrability. Note however, that because using the uniform thickness condition, needed for the reverse Hölder inequality, is valid only when close enough to the lateral boundary, the ratio between $\rho_0$  and the radius $R$ of the cylinder where we want to prove higher integrability affects the constants in the final estimate.
\begin{theorem}[Global higher integrability]
\label{thm:global-high-int}
Let $u\in L^2(0,T;N^{1,2}(\Omega))$ be a parabolic quasiminimizer in $\Omega_T$, where $\Omega$ is such that $X\setminus \Omega$ is uniformly $2$-thick, and that $\eta:[0,T)\times \Omega \mapsto \R$, where $\eta \in W^{1,2}(0,T;N^{1,2}(\Omega))$ and $\eta (x,0)\in N^{1,2}(\Omega)$, sets a parabolic boundary condition for $u$, as described in section \ref{standing assumptions}.  Let $\rho_0$ be the constant from Lemma \ref{reverse hölder lateral}. 

Suppose that we also have $\eta \in W^{1,2+\varepsilon'}(0,T;N^{1,2+\eps'}(\Omega))$ for some positive $\varepsilon'$. Then there exists a positive constant $\varepsilon$ and a positive constant $c=c(c_\mu,c_P,\lambda,K, \max\{1,R/\rho_0\})$, such that for every $Q_R=B_R \times \Lambda_R\subset X\times (-\infty,T)$, we have
\[
\begin{split}
&\left(\frac{1}{\nu(Q_R)}\int_{Q_{R} \cap \Omega_T} g_u^{2+\eps} \,d \nu\right)^\frac{1}{2+\eps}\leq  \l(\frac{c}{\nu(Q_{2R})}\int_{Q_{2R}\cap \Omega_T} g_u^2 \,d \nu\r)^{\frac{1}{2}}\\
&+\left(\frac{c}{\nu(Q_{2R})}\int_{Q_{2R}\cap \Omega_T} g_\eta^{2+\varepsilon}\,d \nu\right)^\frac{1}{2+\eps}+\left(\frac{c}{\nu(Q_{2R})}\int_{Q_{2R}\cap \Omega_T} \left| \frac{\partial \eta}{\partial t}\right|^{2+\varepsilon}\,d \nu\right)^{\frac{1}{2+\eps}}\\
&\qquad+\left(\frac{c}{\mu(B_{2R})}\int_{B_{2R}\cap \Omega} g_\eta^{q+\eps}(x,0)\,d \mu\right)^{\frac{1}{q+\eps}},
\end{split}
\]
where $\eps<\eps'$, $1<q<2$ and $q+\eps<2$.

In case $Q_{2R}$ is such that $B_{2R}\subset \Omega$, we obtain a stronger estimate, in the sense that the second and third term on the right hand side of the above expression can be dropped, and in this case $c=c(c_\mu,c_P,\lambda,K)$. Moreover, in this case we only need to assume that $u$ satisfies the initial condition \eqref{initial condition 2} with some $\eta\in N^{1,2}(\Omega)$.
 
\begin{proof} Assume a parabolic cylinder $Q_{2R}=B_{2R}\times \Lambda_{2R}$. In the case $B_{2R} \cap \Omega = \emptyset$, the claim is true. In the case $B_{2R}\subset \Omega$, then by Lemma \ref{reverse hölder away from boundary}, inequality \eqref{gekaehto} holds with $A=2\lambda$, $f_1=0$ and $f_2=g_\eta(x,0)$ for every $z'$ and $\rho$ such that  $Q_{2\lambda \rho}(z')\subset Q_{2R}$. Theorem \ref{global gehring} then implies the result.

Let then $B_{2R}$ be such that $B_{2R}\cap \Omega \neq \emptyset$ and $B_{2R} \setminus \Omega \neq \emptyset$. We set $A=\max\{6\lambda^2, 
2R/\rho_0\}$, where $\rho_0$ is the constant from Lemma \ref{reverse hölder lateral}. Assume $z'$ and $\rho$ are such that  $Q_{A \rho}(z')\subset Q_{2R}$. In case $B_{2\lambda \rho}(z')\subset \Omega$, we use Lemma \ref{reverse hölder away from boundary}. In case $B_{2\lambda \rho}(z')\setminus\Omega\neq \emptyset$, since necessarily $\rho<\rho_0$, we can use Lemma \ref{reverse hölder lateral}. In both cases, after estimating from above on the right hand side in the former case, we obtain inequality \eqref{gekaehto} with $A=\max\{6\lambda^2, 
2R/\rho_0\}$ and 
\begin{align*}
f_1=g_\eta+\left| \frac{\partial \eta}{\partial t}\right|, \qquad f_2=g_\eta(x,0).
\end{align*}
Theorem \ref{global gehring} now completes the proof.
\end{proof}
\end{theorem}

\def\cprime{$'$} \def\cprime{$'$} \def\cprime{$'$} \def\cprime{$'$}

%\bibliography{citations,citations2}
%\bibliographystyle{alpha}

\label{viimeinensivu}

\end{document}